\documentclass[preprint, 12pt]{elsarticle}

\usepackage{verbatim}
\usepackage{amssymb}
\usepackage{amsmath}
\usepackage{amsthm}
\usepackage{mathabx}
\usepackage{natbib}
\usepackage{enumerate}
\usepackage[active]{srcltx}
\usepackage{color}

\usepackage[utf8]{inputenc}
\usepackage{float}
\usepackage{geometry}
\usepackage{tikz}
 \usepackage[usenames,dvipsnames]{pstricks}
 \usepackage{epsfig}
\usepackage{pst-grad} 

\newtheorem{theorem}{Theorem}[section]
\newtheorem{lemma}[theorem]{Lemma}

\newtheorem{corollary}[theorem]{Corollary}
\newtheorem{obs}[theorem]{Observation}

\newtheorem{prob}[theorem]{Problem}
\newtheorem{claim}{Claim}[theorem]
\newtheorem{cclaim}[theorem]{Claim}
 
\theoremstyle{definition}

\newtheorem{definition}[theorem]{Definition}

\newcommand{\sat}[1]{{#1}\text{-linked}}
\newcommand{\gr}{G=(V,E)}
\newcommand{\lkd}{\text{infinitely linked}}
\newcommand{\ds}[1]{\displaystyle{#1}}

\newif\ifdeveloping

\ifdeveloping
\usepackage[notref,notcite]{showkeys}
\fi

\newif\ifcommented
\commentedtrue

\newcommand{\comm}[1]{}

\ifcommented
\renewcommand{\comm}[1]{
\fbox{\fbox{\begin{minipage}{300pt}#1\end{minipage}}
}}

\fi

\newcommand{\smf}{\hspace{0.008 cm}^\smallfrown}

\newcommand{\mc}[1]{\mathcal{#1}}
\newcommand{\mbb}[1]{\mathbb{#1}}

\newcommand{\mf}[1]{\mathfrak{#1}}
\newcommand{\uhp}{\upharpoonright}

\newcommand{\omg}{{\omega_1}}
\newcommand{\half}{H_{\oo,\oo}}
\newcommand{\halff}[1]{H_{#1,#1}}

\newcommand{\NN}{\mathbb{N}}

\newcommand{\setm}{\setminus}
\newcommand{\empt}{\emptyset}
\newcommand{\subs}{\subset}

\def\<{\left\langle}
\def\>{\right\rangle}
\def\br#1;#2;{\bigl[ {#1} \bigr]^ {#2} }

\newcommand{\oo}{{\omega_1}}

\newcommand{\vv}[1]{\operatorname{V}(#1)}

\newcommand{\nns}[2]{\operatorname{N}_{#1}[#2]}

\newcommand{\gm}[3]{\operatorname{\mf G}_{#3}(#1,#2)}

\newcommand{\ekk}{k}

\newcommand{\Covrel}[3]{#1 \sqsubset  (#2)_{#3}}
\newcommand{\notCovrel}[3]{#1 \not\sqsubset  (#2)_{#3}}

\newcommand{\Covrelst}[3]{#1 \sqsubset^*  (#2)_{#3}}

\definecolor{cthm}{rgb}{0,0,0.7} 
\definecolor{cdef}{rgb}{1,0,0} 
\definecolor{midg}{rgb}{0,0.7,0} 
\definecolor{cemph}{rgb}{0,0,1} 
\definecolor{orange}{rgb}{1,0.3,0}
\definecolor{depg}{rgb}{0,0.5,0} 

\def\oe\{#1,#2\}{#1#2}

\begin{document}

\begin{frontmatter}

\title{Decompositions of edge-colored infinite complete graphs into monochromatic paths}

\author{M\'arton Elekes}
\address
      { Alfr\'ed R\'enyi Institute of Mathematics, 
        Hungarian Academy of Sciences, Re\'altanoda u. 13-15, Budapest
1053, Hungary and E\"otv\"os Lor\'and University, Department of Analysis,
P\'az\-m\'any P. s. 1/c, Budapest 1117, Hungary}
\ead{elekes.marton@renyi.mta.hu}
\ead[url]{http://www.renyi.hu/$\sim$emarci}

\author{D\'aniel T. Soukup}
\address
      {Department of Mathematics and Statistics, University of Calgary, Calgary, Alberta, Canada T2N 1N4}
\ead{daniel.t.soukup@gmail.com}
\ead[url]{http://www.renyi.hu/$\sim$dsoukup}

\author{Lajos Soukup}
\address
      { Alfr\'ed R\'enyi Institute of Mathematics, 
        Hungarian Academy of Sciences, Re\'altanoda u. 13-15, Budapest
1053, Hungary}
\ead{soukup@renyi.hu}
\ead[url]{http://www.renyi.hu/$\sim$soukup}

\author{Zolt\'an Szentmikl\'ossy}

\address
      { E\"otv\"os Lor\'and University, Department of Analysis,
P\'az\-m\'any P. s. 1/c, Budapest 1117, Hungary}
\ead{szentmiklossy@renyi.hu}

\date{\today}

\begin{abstract}
An {\it $r$-edge coloring} of a graph or 
hypergraph $G=(V,E)$  is a 
map $c:E\to \{0, \dots, r-1\}$. 
Extending results of Rado and answering questions of Rado, Gy\'arf\'as and 
S\'ark\"ozy we prove that
\begin{itemize}
 \item the vertex set of every $r$-edge colored 
countably infinite complete $k$-uniform hypergraph  can be partitioned into $r$ 
monochromatic tight paths with distinct colors (a {\em tight path} in a 
$k$-uniform hypergraph is a sequence of distinct vertices
such that every set of $k$ consecutive vertices forms an edge);
\item for all natural numbers $r$  and $k$ there is a natural number $M$ 
such that the vertex  set of  every $r$-edge colored countably infinite  complete graph  
can be  partitioned into $M$ monochromatic $k^{th}$ powers of paths apart 
from a finite set (a {\em $k^{th}$ power of a path} is a sequence $v_0, v_1, \dots$ of
distinct vertices such that $1\le |i-j| \le k$ implies that $\oe\{v_i, v_j\}$ is an 
edge);
\item  the vertex set of every $2$-edge colored 
countably infinite 
complete graph  can be partitioned into 
$4$ monochromatic squares of paths, but not necessarily into $3$;
\item the vertex set of every $2$-edge colored  complete graph on $\omega_1$ 
can be partitioned  into $2$ monochromatic paths with distinct colors.
\end{itemize}

\end{abstract}

\end{frontmatter}

\section{Introduction}

In this paper we will follow the conventions of \cite{Gy2}.

 Our goal is to find partitions of the vertex sets of edge-colored infinite 
graphs and hypergraphs into nice monochromatic subgraphs.  
In particular, we are interested 
in partitioning the vertices of complete graphs and hypergraphs into 
monochromatic paths and powers of paths. 

An {\it $r$-edge coloring} of a graph or hypergraph $G=(V,E)$  is a 
map $c:E\to \{0, \dots, r-1\}$ where $r \in \NN \setminus \{0\}$. 
Investigations began in the 1980s with a result of Rado \cite{R}  implying that
the vertex set of every $r$-edge colored countably infinite complete  graph  can be
partitioned into $r$ monochromatic paths with distinct colors; this includes the possibility that some of the paths are empty or single vertices. We will abbreviate this result as
\begin{equation}
 \Covrel{K_{\NN}}{\mf{Path},\dots ,\mf{Path}}{r}.
\end{equation}
This notation is in direct analogy with the usual ``arrow notation" from partition calculus and Ramsey-theory popularized by P. Erd\H os et al \cite{partcalc}; we give the exact definition in Section \ref{notsec}.

In Section \ref{sc:hyper}, answering a question of Gy\'arf\'as and S\'ark\"ozy from \cite{GyS1} 
we extend this  result  for hypergraphs  by proving
that  the vertex set of every $r$-edge colored countably infinite 
complete $k$-uniform hypergraph  can be partitioned into $r$ 
monochromatic tight paths with distinct colors (Theorem \ref{tm:tight_path}): 
\begin{equation}
 \Covrel{K^k_{\NN}}{\mf{TightPath},\dots ,\mf{TightPath}}{r}.
\end{equation}

Furthermore, Erd\H os, Gy\'arf\'as and Pyber \cite{EGyP} conjectured that 
the vertices of every $r$-edge colored
finite complete graph can be covered with $r$ disjoint monochromatic cycles.

This conjecture was disproved by Pokrovskiy \cite{Po2}.
However, the case $k=2$ of Theorem \ref{tm:tight_path}(2) yields that 
 the corresponding version 
of the conjecture above holds for 
countably infinite graphs:
{\em Given an $r$-edge coloring of
$K_\NN$, we can partition the vertices into $r$ 
disjoint cycles and $2$-way infinite paths of distinct colors.}

In Section \ref{sc:power}, we  prove that for all natural numbers $r$ and $k$ there is a natural  number $M$ 
such that the vertex set of every $r$-edge colored complete graph on $\NN$ can be 
partitioned into $M$ monochromatic $k^{th}$ powers of paths apart 
from a finite set (Theorem \ref{tm:fonat1}):
\begin{equation}
 \Covrelst{K_{\NN}}{\text{$k^{\mf th}-\mf{Power\  of\  Path}$}}{r,M}.
\end{equation}

 Using a recent result of Pokrovskiy on finite graphs 
we show that the vertex set of every $2$-edge colored 
complete graph on $\NN$ can be partitioned into $4$ monochromatic squares of paths:
\begin{equation}
 \Covrel{K_{\NN}}{\mf{PathSquare}}{2,4}.
\end{equation}

Finally, in Section \ref{longpath}, we give a partial answer to a question of
Rado from \cite{R} (the definitions are postponed to the section): 
the vertex set of every $2$-edge colored complete graph on $\omega_1$ 
(the smallest uncountable vertex set) can be partitioned into
$2$ monochromatic paths of distinct colors:
\begin{equation}
 \Covrel{K_{\omega_1}}{\mf{Path},\mf{Path}}{2}.
\end{equation}

The paper ends with the short Section \ref{sc:problems} on further results (without proofs) and open problems.

\section{Notations, preliminaries}\label{notsec}


The cardinality of a set $X$ is denoted by $|X|$. For a set $X$ and $k \in \NN$ 
we will denote the set of $k$-element subsets of $X$ by $[X]^k$. The set of all 
subsets of $X$ is denoted by $\mathcal{P}(X)$. 

For a graph $\gr$ and $v\in V$ we write 
$$N_G(v)=\{w\in V:\oe\{v,w\}\in E\},$$  and 
for $F\subs V$
$$N_G[F]=\{w\in V:  \oe\{v,w\}\in E\text{ for all }v\in F\}.$$
In particular, $N_G[\emptyset]=V$.


Let $c:E\to \{0, \dots, r-1 \}$ be an $r$-edge coloring of a graph $\gr$.

For $ v \in V$ and $i < r$ also let $$N_G(v,i)=\{w\in N_G(v):c(vw)=i\},$$ and 
for $F \subset V$ and $i < r$ let 
$$N_G[F,i]=
\{w\in V: c(vw)=i\text{ for all }v\in F\}.$$
In particular, $N_G[\empt, i]=V$ for all $i<r$.

As we always work with a fixed coloring, this notation will lead
to no misunderstanding (and sometimes we will even drop the subscript $G$). 

We will use $K_\NN$ to denote
the complete graph on $\NN$. A \emph{path} in a
graph is a finite or one-way infinite sequence of distinct vertices such
that each pair of consecutive vertices is connected by an edge. If $P$ is a
finite path and $Q$ is a disjoint path such that the end-point of $P$ is
connected by an edge with the starting point of $Q$ then $P^\frown Q$ denotes
their concatenation. We say that $Q$ \emph{end extends} $P$ if $P$ is an
initial segment of $Q$.

\begin{definition} Let $\gr$ be a graph and $A\subs V$.
We say that $A$ is \emph{$\lkd$} iff there are infinitely many 
vertex disjoint finite paths between any two distinct points of $A$. We say that $A$ is
\emph{infinitely connected} iff there are infinitely many vertex disjoint finite
paths \emph{inside $A$} between any two distinct points of $A$.
\end{definition}

%

\noindent\emph{Remark.} 
Clearly 
$A$ is $\lkd$ iff for every two distinct members $v$ and $w$ of $A$ and every finite set $F \subset
V(G) \setminus \{v, w \}$ there is a path connecting the two points and
avoiding $F$. Similarly, $A$ is infinitely connected if we can additionally
require that the path is \emph{inside $A$}.

\medskip

If we fix an edge coloring $c$ of $G$ with $r$ colors, $i < r$, $\mc P$ is a
graph property (e.g. being a path, being infinitely connected...) and $A\subs V$
then we say that $A$  \emph{has property $\mc P$ in color $i$} (with respect to
$c$) iff $A$ has property $\mc P$  in the graph $(V,c^{-1}(i))$. In particular,
by a \emph{monochromatic path} we mean a set $P$ which is a path in some color.

By convention, the empty set and singletons  are  monochromatic paths in any color. 

\begin{lemma}\label{uftrick} Let $G = (V, [V]^2)$ be a complete countably 
infinite graph. Given any edge
coloring $c: [V]^2 \to \{0, \dots, r-1 \}$, there is a function $d_c: V \to 
\{0, \dots, r-1 \}$ and an integer  $i_c < r$ such that the sets 
$V_i=d_c^{-1}\{i\}$ satisfy:
$$N[F,i]\cap V_{i_c} \text{ is infinite for all } i < r \text{ and finite set } 
F\subs V_i.
$$ 

In particular, $V_i$ is $\lkd$ in color $i$ for all $i < r$ and $V_{i_c}$
is infinitely connected in color $i_c$. 
\end{lemma}

We remark here that the empty set and one element vertex sets are 
$\lkd$ in any color.

\begin{proof} Let $U$ be a non-trivial ultrafilter on $V$, see e.g. 
\cite{kunen}. (In other words, take a finitely additive $0/1$-measure on $\mc P(V)$ 
assigning measure $0$ to singletons, and let $U$ be the class of sets of measure 
$1$.) For $i < r$ define $V_i=\{v\in
V:N(v,i)\in U\}$ (e.g. $d_c\uhp V_i\equiv i$), and let $i_c$ be the unique 
element of $\{0, \dots, r-1\}$ with
$V_{i_c} \in U$. 

If $i<r$, and $v$ and $w$ are distinct elements of $V_i$, then 
$N_G[v]\cap N_G[w]\in U$, so the set $N_G[v]\cap N_G[w]$ is infinite,
and $vuw$ is a monochromatic path in color $i$ 
for all $u\in N_G[v]\cap N_G[w]$. So $V_i$ is $\lkd$ in color $i$.

If $i=i_c$, then even the set $N_G[v]\cap N_G[w]\cap V_{i_c}\in U$, and so 
$vuw$ is a monochromatic path in color $i$ {\em inside $V_{i_c}$}
for all $u\in N_G[v]\cap N_G[w]\cap V_{i_c}$. So $V_{i_c}$ is infinitely connected 
in color $i$.
\end{proof}

The next lemma looks slightly technical at first sight. However, note that for our first application, that is for the proof of Rado's theorem we can ignore the sets $A_j$, as well as the last clause.

\begin{lemma}\label{simpath} Suppose that $\gr$ is a countably infinite graph
and $c$ is an edge coloring. Suppose that $\{C_j:j<k\}$ is a finite family of
subsets of $V$ and that each $C_j$ is $\lkd$ in some color
$i_j$. Moreover, for $j<k$ let $A_j \subseteq C_j$ be arbitrary subsets. 

Then we can find disjoint vertex sets $P_j$ so that 
\begin{enumerate}[$($a$)$]
\item $P_j$ is a path (either finite or one-way infinite) in color $i_j$
for all $j<k$,
\item if $A_j$ is infinite then so is $A_j\cap P_j$,
\item $\bigcup\{P_j:j<k\} \supset \bigcup\{C_j:j<k\}$.
 \end{enumerate}
Moreover, if $C_j$ is infinite then we can choose the first point of $P_j$ freely from $C_j$.
%
%
\end{lemma}

\begin{proof}
Let $v_0, v_1, \dots$ be a (possibly finite) enumeration of $\bigcup\{C_j:j<k\}$.

 For all the infinite $C_j$, fix distinct $x_j \in C_j$ as starting points for the $P_j$s. We define
disjoint finite paths $\{P_j^n:j<k\}$ by induction on $n \in \NN$ so that 
 \begin{enumerate}[(i)]
  \item $P_j^{n}$ is a path of color $i_j$ with first point $x_j$,
  \item $P_j^{n+1}$ end extends $P_j^{n}$ (as a path of color $i_j$),
\item the last point of the path $P_j^{n}$ is in $C_j$,
\item  if $A_j$ is infinite then the last point of $P_j^{2n}$ is in $A_j$,
 \end{enumerate}
for all $j<k$, and
 \begin{enumerate}[(i)]\setcounter{enumi}{4}
\item if $v_n \notin \bigcup_{j<k} P^{2n}_j$ and $v_n \in C_j$ then $v_n$ is the last point of $P_j^{2n+1}$.
 \end{enumerate}

 It should be easy to carry out this induction applying that each
$C_j$ is infinitely linked in color $i_j$.
Finally, we let $P_j=\cup \{P_j^{n}:n\in \NN\}$ for $j<k$ which finishes the proof.
\end{proof}

In particular, we have the following trivial corollary:

\begin{corollary}\label{trivpath} 
If a countable set of vertices $C$ is $\lkd$ then it is covered by a single one-way infinite path.
\end{corollary}  

More importantly, the above lemmas yield

\begin{theorem}[R. Rado \cite{R}]\label{rado} For every $r$-edge coloring of
$K_\NN$ we can partition the vertices into $r$
disjoint paths of distinct colors.
\end{theorem}
\begin{proof} Apply Lemma \ref{uftrick} and find a partition $\NN = \{V_i:i < r\}$ so that each $V_i$ is $\lkd$ in color $i$. Now apply Lemma
\ref{simpath} with  $C_i = V_i$ (and $A_i = \emptyset$) to get the desired partition into monochromatic paths.
\end{proof}

To abbreviate the formulation of certain result we introduce the following notation.

\begin{definition}\label{df:arrow}
Let $G$ be a graph and
$\mf F$ be a class of graphs.
We write 
\begin{equation}\label{eq:arrow}
\Covrel{G}{\mf F}{r,m} 
\end{equation}
if given any $r$-edge coloring $c:E(G)\to \{0, \dots, r-1 \}$   the vertex set of $G$ 
can be partitioned into $m$ monochromatic elements of $\mf F$.

We write 
\begin{equation}\label{eq:arrow2new}
\Covrel{G}{\mf F,\mf F,\dots,\mf F}{r} 
\end{equation}
if given any $r$-edge coloring $c:E(G)\to \{0, \dots, r-1 \}$   the vertex set of $G$ 
can be partitioned into $r$ monochromatic elements of $\mf F$
 in \emph{distinct colors}.

In particular,   $\Covrel{G}{\mf {Path}}{r,m}$ holds if 
given any $r$-edge coloring $c$ of $G$ the vertex set of $G$ 
can be partitioned into $m$ monochromatic paths.

We write $\sqsubset^*$ instead of $\sqsubset$ if we can partition the vertex set \emph{apart 
from a finite set}.
\end{definition}

Using our new notation,  Theorem \ref{rado}
can be formulated as follows:
\begin{equation}
 \Covrel{K_{\NN}}{\mf{Path},\dots ,\mf{Path}}{r}.
\end{equation}

\section{Partitions of hypergraphs}\label{sc:hyper}

In this section, we briefly look at a generalization of Rado's result, Theorem
\ref{rado} above, to hypergraphs. Let $k \in \NN \setminus \{0\}$.

\begin{definition}
  A {\em loose path} in a 
$k$-uniform hypergraph is a finite or one-way infinite sequence of edges, $e_1,e_2,\dots$ such that
$|e_i\cap e_{i+1}|=1$ for all $i$, and $e_i\cap e_j=\empt$ for all $i, j$ with $i+1< j$.

A {\em tight  path} in a 
$k$-uniform hypergraph is a finite or one-way infinite sequence of distinct vertices
such that every set of $k$ consecutive vertices forms an edge.

\end{definition}

\noindent\emph{Remark.} Occasionally, we will refer to loose and tight \emph{cycles} and \emph{two-way infinite paths} as well, with the obvious analogous definitions.

 The following result was proved recently:

\begin{theorem}[A. Gy\'arf\'as, G. N. S\'ark\"ozy {\cite[Theorem 3.]{GyS1}}]
  Suppose that the edges of a countably infinite complete $k$-uniform hypergraph
are colored with $r$  colors. Then the vertex set can be
partitioned into monochromatic
finite or one-way infinite loose paths of distinct colors.

\end{theorem}
 
In the introduction of \cite{GyS1}, the authors asked if one can find a partition into tight paths instead of loose ones.  We prove the following:

\begin{theorem}\label{tm:tight_path}
Suppose that the edges of a countably infinite complete $k$-uniform hypergraph
are colored with $r$  colors. Then 

(1) the vertex set can be partitioned into monochromatic
finite or one-way infinite  tight paths of distinct colors,

(2) the vertex set can be partitioned into monochromatic
tight cycles and two-way infinite  tight paths of distinct colors.
\end{theorem}

\begin{proof}
 (1) Note that the case of $k=2$ is Rado's Theorem \ref{rado} above; we will imitate his original proof here. 

Let $c:\br {\NN};k;\to \{0, \dots, r-1 \}$.  A set $T\subs \{0, \dots, r-1 \}$ of colors is called {\em perfect} 
iff there are disjoint finite  subsets $\{P_t:t\in T\}$ of ${\NN}$ and 
an infinite set $A\subs \NN \setminus \bigcup_{t \in T} P_t$ such that for all $t\in T$
\begin{enumerate}[(a)]
 \item $P_t$ is a tight path in color $t$,
\item if $1\le i<k$ and $x$ is the set of the last $i$ vertices from the tight path $P_t$ and 
$y\in \br A;k-i;$, then $c(x\cup y)=t$.
\end{enumerate}
Since $\empt$ is perfect, we can consider a perfect set $T$ of colors with
maximal number of elements.   

\begin{claim}
If the vertex disjoint finite tight paths $\{P_t:t\in T\}$  and the infinite set 
$A$ satisfy (a) and (b) then for all $v\in {\NN}\setm \bigcup_{t\in T}P_t$ 
there is a color $t'\in T$, a finite sequence $v_1,v_2,\dots, v_{k-1}$ from
$A$,
and an infinite set $A'\subs A$ such that the tight paths
\begin{equation}
 \bigl\{P_t:t\in T\setm \{t'\}\bigr\}\cup \{P_{t'}\hspace{0.005 cm}^\frown (v_1, v_2,\dots, v_{k-1},v)\}
\end{equation}
and $A'$   satisfy (a) and (b) as well.
\end{claim}
\begin{proof}[Proof of the Claim]
Define a new coloring $d:\br A;k-1;\to \{0, \dots, r-1 \}$ by the formula $d(x)=c(x\cup\{v\})$.
By Ramsey's Theorem, there is an infinite $d$-homogeneous set $B\subs A$ in some
color  $t'$. Then $t'\in T$, since otherwise $T\cup \{t'\}$ would be a bigger perfect
set witnessed by $P_{t'}=\{v\}$,$\{P_t:t\in T\}$  and $B$.

Now pick distinct $v_1,v_2, \dots, v_{k-1}$ from $B$ and let $A'=B\setm \{v_1, \dots,
v_{k-1},v\}$.  
\end{proof}

Finally, by applying the claim repeatedly, we can cover the vertices with $|T|$ tight paths of distinct colors. 

\medskip
\noindent 
(2) 
Let $c:\br {\NN};k;\to \{0, \dots, r-1 \}$. 
Write $V_{-1}={\NN}$.
Using Ramsey's Theorem, by induction on $n \in \NN$ choose 
$d(n) < r$ and $V_n\in  \br V_{n-1};{\NN};$ such that 
\begin{equation}
  c(\{n\}\cup O)=d(n)\text{ for all $O\in \br V_n;k-1;$.}
\end{equation}
For $i < r$ let 
\begin{equation}
A_i=\{n\in \NN:  d(n)=i\}. 
\end{equation}

Let $K=\{i < r : A_i\text{ is finite}\}$.
By induction on $i\in K$ we will define tight cycles $\{P_i : i \in K\}$ such that 
$$\ds{\bigcup_{i'<i, i' \in K} A_{i'}\subseteq \bigcup_{i'<i, i' \in K} P_{i'}}$$ while some of the $P_i$'s might be empty.

Assume that $\{P_{i'}:i'<i, i' \in K\}$ is defined and suppose $i\in K$. 
Enumerate $A_i\setm \bigcup_{i'<i, i' \in K} P_{i'}$ as 
$\{x_i^{j}:j<t\}$.


Choose disjoint $k-1$ element sets
\begin{equation}
 Y^j_i\subseteq 
\bigcap_{j<t}V_{x_i^{j}} \setminus \bigcup_{i'<i, i' \in K} P_{i'} \text{ for } j<t.
\end{equation}

Consider an ordering $\prec_i$  on  $\ds{P_{i}=\{x^j_i:j<t\}\cup\bigcup_{j<t}Y^j_i}$
such that 
\begin{equation}
 x^0_i\prec_i Y^0_i\prec_i x^1_i\prec_i Y^1_i\prec_i \dots \prec_i x_i^{t-1}\prec_i Y_i^{t-1}.
\end{equation}

Then $\prec_i$ witnesses that $P_i$ is a tight cycle in color $i$.

Now, let 
\begin{equation}
  P=\bigcup_{i\in K} P_i
\end{equation}

and 
for each $i\in \{0, \dots, r-1 \}\setm K$  we define  a 2-way infinite tight path $P_i$ as follows.

By induction,  for every integer $z\in \mbb Z$ and $i \in \{0, \dots, r-1 \}\setm K$ choose disjoint 
sets 
 $\{x^z_i\}\in [A_{i}\setm P]^1$ and $Y^z_i\in \br {\NN}\setm P;k-1;$
such that
\begin{equation}
 Y^z_i\subs  V_{x^z_i}\cap V_{x_i^{z+1}} 
\end{equation}
and 
\begin{equation}
\bigcup_{i\in \{0, \dots, r-1 \}\setm K}  A_i\subs  P\cup  \bigcup \{\{x^z_i\}, Y^z_i:
i\in \{0, \dots, r-1 \}\setm K, z\in \mbb Z\}. 
\end{equation}

Consider an ordering $\prec_i$  on  $P_{i}=\{x^z_i:z\in \mbb Z\}\cup\bigcup_{z\in \mbb Z}Y^z_i$
such that 
\begin{equation}
\ldots \prec_i Y^{-2}_i \prec_i x^{-1}_i \prec_i Y^{-1}_i   \prec_i x^0_i\prec_i Y^0_i\prec_i x^1_i\prec_i Y^1_i\prec_i 
\ldots
\end{equation}

Then $\prec_i$ witnesses that $P_i$ is a $2$-way infinite tight path in color $i$.
\end{proof}

\section{Covers by $k^{th}$ powers of paths}\label{sc:power}

Our aim is to prove a stronger version of Rado's theorem; in order to state
this result we need the following

\begin{definition}
 Suppose that $\gr$ is a graph and $k \in \NN \setminus\{0\}$. The {\em $k^{th}$ power of $G$} is the
graph $G^k=(V,E^k)$ where $\{v,w\}\in E^k$ iff there is a finite path of length
$\leq k$ from $v$ to $w$.
\end{definition}

We will be interested in partitioning an edge colored copy of $K_\NN$ into
finitely many monochromatic \emph{$k^{th}$ powers of paths}. 

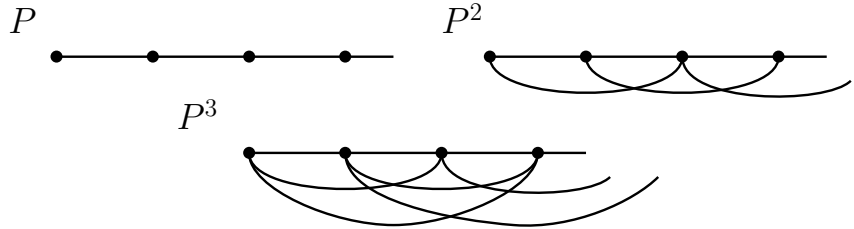
\begin{figure}[H]%
\centering
\psscalebox{0.8 0.8} 
{
\begin{pspicture}(1,-5.8085423)(14.014142,-2)
\psline[linecolor=black, linewidth=0.04](0.8,-2.9114575)(6.4,-2.9114575)
\psdots[linecolor=black, dotsize=0.2](0.8,-2.9114575)
\psdots[linecolor=black, dotsize=0.2](2.4,-2.9114575)
\psdots[linecolor=black, dotsize=0.2](4.0,-2.9114575)
\psdots[linecolor=black, dotsize=0.2](5.6,-2.9114575)
\psdots[linecolor=black, dotsize=0.2](9.6,-2.9114575)
\psdots[linecolor=black, dotsize=0.2](11.2,-2.9114575)
\psdots[linecolor=black, dotsize=0.2](12.8,-2.9114575)
\psline[linecolor=black, linewidth=0.04](8.0,-2.9114575)(13.6,-2.9114575)
\psdots[linecolor=black, dotsize=0.2](8.0,-2.9114575)
\psbezier[linecolor=black, linewidth=0.04](8.0,-2.9114575)(8.0,-3.7114575)(11.2,-3.7114575)(11.2,-2.9114575)
\psbezier[linecolor=black, linewidth=0.04](9.6,-2.9114575)(9.6,-3.7114575)(12.8,-3.7114575)(12.8,-2.9114575)
\psbezier[linecolor=black, linewidth=0.04](11.2,-2.9114575)(11.2,-3.7114575)(13.6,-3.7114575)(14.0,-3.3114576)
\psdots[linecolor=black, dotsize=0.2](5.6,-4.5114574)
\psdots[linecolor=black, dotsize=0.2](7.2,-4.5114574)
\psdots[linecolor=black, dotsize=0.2](8.8,-4.5114574)
\psline[linecolor=black, linewidth=0.04](4.0,-4.5114574)(9.6,-4.5114574)
\psdots[linecolor=black, dotsize=0.2](4.0,-4.5114574)
\psbezier[linecolor=black, linewidth=0.04](4.0,-4.5114574)(4.0,-5.3114576)(7.2,-5.3114576)(7.2,-4.5114574)
\psbezier[linecolor=black, linewidth=0.04](5.6,-4.5114574)(5.6,-5.3114576)(8.8,-5.3114576)(8.8,-4.5114574)
\psbezier[linecolor=black, linewidth=0.04](7.2,-4.5114574)(7.2,-5.3114576)(9.6,-5.3114576)(10.0,-4.9114575)
\psbezier[linecolor=black, linewidth=0.04](4.0,-4.5114574)(4.0,-5.1114573)(5.4,-5.7114577)(6.4,-5.7114577)(7.4,-5.7114577)(8.8,-5.1114573)(8.8,-4.5114574)
\psbezier[linecolor=black, linewidth=0.04](5.6,-4.5114574)(5.6,-5.3114576)(7.4029455,-5.634761)(8.4,-5.7114577)(9.397055,-5.788154)(10.4,-5.3114576)(10.8,-4.9114575)
\rput[bl](0.0,-2.5114574){\Large $P$}
\rput[bl](7.2,-2.5114574){\Large $P^2$}
\rput[bl](2.8,-4.1114573){\Large $P^3$}
\end{pspicture}
}

\caption{Powers of paths.}%

\end{figure}

We will investigate this problem by
introducing the following game.

\begin{definition}\label{def:games}
Assume that $H$ is a graph, $W\subs V(H)$ and $\ekk\in {\NN}$. 
The game $\gm HW\ekk$ is played by two players, Adam and Bob,   as follows.
The  players choose disjoint finite subsets of $V(H)$ alternately:
\begin{displaymath}
 A_0, B_0, A_1, B_1, \dots
\end{displaymath}
Bob wins the game $\gm HW\ekk$ iff 
\begin{enumerate}[(A)]
 \item $W\subs \bigcup_{i\in {\NN}}A_i\cup B_i$, and  
\item $H[\bigcup_{i\in {\NN}}B_i]$ contains the  $\ekk^{th}$
 power of a (finite or one way infinite) Hamiltonian path (that is, a path covering
all the vertices).
\end{enumerate}
\end{definition}

For $k=1$, we have the following

\begin{obs}\label{winnerpath} If $H=(V,E)$ is a countable graph and $W\subs V$
then the following are equivalent:
\begin{enumerate}
\item $W$ is $\lkd$,
\item Bob wins $\gm HW 1$.
\end{enumerate}
\end{obs}
\begin{proof} 
(1) $\Rightarrow$ (2): By our assumption, Bob can always connect an uncovered
point of $W$ to the end-point of the previously constructed path while avoiding  vertices played so far.
This shows the existence of a winning strategy for Bob.

(2) $\Rightarrow$ (1): Fix any two distinct points $v,w\in W$ and a finite set $F\subs
V\setm \{v,w\}$. Let Adam start with $A_0=F$ and continue with $A_i=\emptyset$;
the Hamiltonian path $P$ constructed by Bob's strategy will go through $a$ and
$b$ while $P\cap F=\emptyset$.
\end{proof}

Now, we show how to produce a partition of the vertices into $k^{th}$ powers of
paths using winning strategies of Bob:

\begin{lemma}\label{parallel} Suppose that $H=(V,E)$, $V=\bigcup\{W_i:i<M\}$ with
$M\in \NN$ and let $H_i=(V, E_i)$ for some $E_i\subs E$. If Bob wins $\gm
{H_i} {W_i} \ekk$ for all $i<M$ then $V$ can be partitioned into $\{P_i:i<M\}$
so that $P_i$ is a $k^{th}$ power of a path in $H_i$.
\end{lemma}
\begin{proof}
 We will conduct $M$ games simultaneously as follows: the plays of Adam and Bob
in the $i^{th}$ game will be denoted by $A^i_0,B^i_0, A^i_1,B^i_1,\dots$ for $i<M$.
Let $\sigma^i$ denote the winning strategy for Bob in $\gm {H_i} {W_i} \ekk$, that is, 
if we set $B^i_{n}=\sigma^i(A^i_0,B^i_0, \dots, A^i_n)$ then Bob wins the
game.
 
 Now, we define $A^i_n,B^i_n$ by induction using the lexicographical ordering 
$<_{lex}$
on
$\{(n,i):n\in\NN, i<M\}$. First, let $A^0_0=\emptyset$ and
$B^0_0=\sigma^0(A^0_0)$. In general, assume that 
$A^j_m$ and $B^j_m$ are defined for $(m,j)<_{lex}(n,i)$, 
and 
we let 
\begin{align}\label{eq:Ain}
A^i_n=\bigcup\{B^j_m:(m,j)<_{lex}(n,i)\}\setm \Big(\bigcup\{A^i_m,B^i_m:m<n\}\Big) 
\end{align}
and 
 $$B^i_n=\sigma^i(A^i_0,B^i_0, \dots, A^i_n).$$
 
 One easily checks that the above defined plays are valid; indeed, 
for a fix $i<M$
the finite
sets  $\{A^i_n,B^i_n:n\in \mbb N\}$  defined above are disjoint.
 
Next, let $P_i=\bigcup \{B^i_n:n\in
\NN\}$ for $i<M$. 
As Bob
wins the $i^{th}$ game we have that $P_i$ is a $k^{th}$ power of path in $H_i$.
Note that $P_i\cap P_j=\emptyset$ if $i\neq j<M$. Indeed, if $(m,j)<_{lex}(n,i)$, then 
\begin{align}
 B^i_n\cap B^j_m\subs B^i_n\cap (A^i_n\cup \Big(\bigcup\{A^i_m,B^i_m:m<n\}\Big)=\empt   
\end{align}
by (\ref{eq:Ain}).

To
finish the proof, we prove 
\begin{align}
V=\{P_i:i<M\}. 
\end{align}
%
 Indeed, first note that $W_i\subs \bigcup_{n\in {\NN}}A^i_n\cup B_n^i$ as
Bob wins the $i^{th}$ game and hence $$V= \bigcup_{n\in {\NN},i<M}A^i_n\cup
B_n^i.$$
 Second,  by (\ref{eq:Ain}), we have  
\begin{displaymath}
A^i_n\subs \bigcup\{B^j_m: (m,j)<_{lex}(n,i)\} 
\end{displaymath}
and so 
$$\bigcup_{n\in {\NN},i<M}A^i_n\subs \bigcup_{n\in
{\NN},i<M}B^i_n$$ and hence $V=\{P_i:i<M\}$.
\end{proof}

The next theorem provides conditions under which Bob has a winning strategy:

\begin{theorem}\label{tm:winnig}
Assume that $H$ is a countably infinite graph, $W\subs V(H)$ is non-empty and $\ekk\in {\NN}$.
If there are subsets $W_0,\dots, W_{\ekk}$ of $V(H)$ such that 
$W_0=W$ and 
\begin{equation}\label{eq:winnig}
W_{j+1}\cap N_H[F]\text{ is infinite for each } j<\ekk \text{ and finite } F\subs  
\bigcup\nolimits_{i\le j}W_i 
\end{equation}
then Bob wins $\gm HW\ekk$. 
\end{theorem}

\begin{proof}
%
We can assume that $\vv H={\NN}$.

Consider first the easy case when  $W_0$ is finite.
Adam plays a finite set $A_0$ in the first round. 
Write $N=|W_0\setm A_0|$.
 Let Bob play $B_0=W_0\setm A_0=\{b_{n,0}:n<N\}$.
In the $j^{th}$ round for $1\le j\le k$, let Bob play 
 an $N$-element set 
\begin{align}\label{eq:fromWjfin}
  B_j=\{b_{n,j}:n<N\}\subs W_j\cap N_H\big[\bigcup\nolimits_{i<j}B_{i}\big] 
\end{align}
which avoids all previous choices, i.e. $B_j \cap \bigcup \{A_{i'},B_{i}:i'\leq j,i<j\}=\emptyset$.
 For $j>k$ let Bob play $B_j=\empt$.

We claim that 
\begin{enumerate}[(A)]
\item  $W_0\subseteq \bigcup\{A_n, B_n:n\in \mbb N \}$, and
 \item $P=\{b_{n,j}:n< N, j\le k\}$ is the $k^{th}$-power of a path.
\end{enumerate}

(A) is clear because $W_0\subseteq A_0\cup B_0$.

To check (B) consider the lexicographical order of the indexes. Let $(m,i)\neq (n,j)\in \{0, \dots, N-1\}\times \{0,\dots, k\}$. Then $b_{m,i}$ and  $b_{n,j}$ are  the  $((k+1)m+i)^{th}$ and  
$((k+1)n+j)^{th}$ elements, respectively, in the lexicographical order.

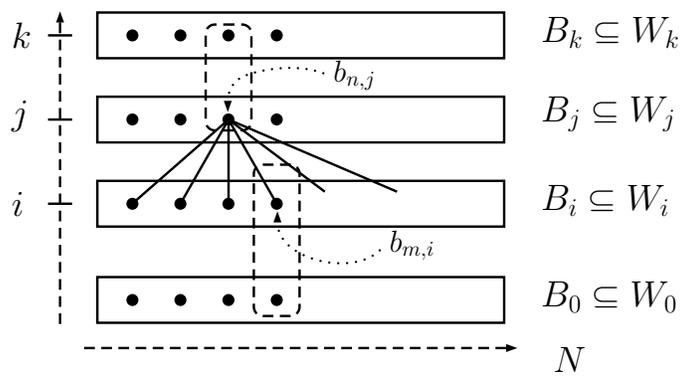
\begin{figure}[H]
\centering
\psscalebox{0.8 0.8} 
{
\begin{pspicture}(0,-3.0)(10.74,3.0)
\psline[linecolor=black, linewidth=0.04, linestyle=dashed, dash=0.17638889cm 0.10583334cm, arrowsize=0.05291666666666667cm 2.0,arrowlength=1.4,arrowinset=0.0]{->}(0.8,-2.2)(0.8,3.0)
\psline[linecolor=black, linewidth=0.04, linestyle=dashed, dash=0.17638889cm 0.10583334cm, arrowsize=0.05291666666666667cm 2.0,arrowlength=1.4,arrowinset=0.0]{->}(1.2,-2.6)(8.4,-2.6)
\psframe[linecolor=black, linewidth=0.04, dimen=outer](8.2,3.0)(1.4,2.2)
\psframe[linecolor=black, linewidth=0.04, dimen=outer](8.2,0.2)(1.4,-0.6)
\psframe[linecolor=black, linewidth=0.04, dimen=outer](8.2,1.6)(1.4,0.8)
\psframe[linecolor=black, linewidth=0.04, dimen=outer](8.2,-1.4)(1.4,-2.2)
\rput[bl](8.8,1.0){\Large $B_j\subseteq W_j$}
\rput[bl](8.8,-0.4){\Large $B_i\subseteq W_i$}
\rput[bl](8.8,-2.0){\Large $B_0\subseteq W_0$}
\rput[bl](8.8,2.4){\Large $B_k\subseteq W_k$}
\psdots[linecolor=black, dotsize=0.2](2.0,2.6)
\psdots[linecolor=black, dotsize=0.2](2.8,2.6)
\psdots[linecolor=black, dotsize=0.2](3.6,2.6)
\psdots[linecolor=black, dotsize=0.2](4.4,2.6)
\psdots[linecolor=black, dotsize=0.2](2.0,1.2)
\psdots[linecolor=black, dotsize=0.2](2.8,1.2)
\psdots[linecolor=black, dotsize=0.2](3.6,1.2)
\psdots[linecolor=black, dotsize=0.2](4.4,1.2)
\psdots[linecolor=black, dotsize=0.2](2.0,-0.2)
\psdots[linecolor=black, dotsize=0.2](2.8,-0.2)
\psdots[linecolor=black, dotsize=0.2](3.6,-0.2)
\psdots[linecolor=black, dotsize=0.2](4.4,-0.2)
\psdots[linecolor=black, dotsize=0.2](2.0,-1.8)
\psdots[linecolor=black, dotsize=0.2](2.8,-1.8)
\psdots[linecolor=black, dotsize=0.2](3.6,-1.8)
\psdots[linecolor=black, dotsize=0.2](4.4,-1.8)
\psline[linecolor=black, linewidth=0.04](1.0,1.2)(0.6,1.2)
\psline[linecolor=black, linewidth=0.04](1.0,-0.2)(0.6,-0.2)
\rput[bl](0.0,-0.4){\Large{$i$}}
\rput[bl](0.0,1.0){\Large{$j$}}
\psline[linecolor=black, linewidth=0.04](1.0,2.6)(0.6,2.6)(0.6,2.6)
\rput[bl](0.0,2.4){\Large{$k$}}
\rput[bl](9.0,-3.0){\Large{$N$}}
\psline[linecolor=black, linewidth=0.04](3.6,1.2)(2.0,-0.2)
\psline[linecolor=black, linewidth=0.04](3.6,1.2)(2.8,-0.2)
\psline[linecolor=black, linewidth=0.04](3.6,1.2)(3.6,-0.2)
\psline[linecolor=black, linewidth=0.04](3.6,1.2)(4.4,-0.2)
\psline[linecolor=black, linewidth=0.04](3.6,1.2)(5.2,0.0)
\psline[linecolor=black, linewidth=0.04](3.6,1.2)(6.4,0.0)
\psframe[linecolor=black, linewidth=0.04, linestyle=dashed, dash=0.17638889cm 0.10583334cm, dimen=outer, framearc=0.4](4.0,2.8)(3.2,1.0)
\psframe[linecolor=black, linewidth=0.04, linestyle=dashed, dash=0.17638889cm 0.10583334cm, dimen=outer, framearc=0.4](4.8,0.4695652)(4.0,-2.0869565)
\psbezier[linecolor=black, linewidth=0.04, linestyle=dotted, dotsep=0.10583334cm, arrowsize=0.05291666666666667cm 2.0,arrowlength=1.4,arrowinset=0.0]{->}(5.175,1.9701645)(4.7125,2.1434405)(4.1125,2.0284483)(3.8625,1.925)(3.6125,1.8215517)(3.5875,1.5732758)(3.5875,1.325)
\rput[bl](5.35,1.6875){\large{$b_{n,j}$}}
\psbezier[linecolor=black, linewidth=0.04, linestyle=dotted, dotsep=0.10583334cm, arrowsize=0.05291666666666667cm 2.0,arrowlength=1.4,arrowinset=0.0]{->}(6.1125,-0.975)(5.65,-1.425)(4.4125,-1.125)(4.4125,-0.325)
\rput[bl](6.275,-1.1375){\large{$b_{m,i}$}}
\end{pspicture}
}
\caption{$b_{n,j}$ and its $k$ successors.}%
\label{}%
\end{figure}

Assume that $|((k+1)m+i)- ((k+1)n+j)|\le k$;
then $i\ne j$ and, without loss of generality,  we can suppose that $i<j$. Then we have 
$b_{m, i}\in \bigcup\nolimits_{i'<j}B_{i'}$, so
$b_{n,j}\in N_H(b_{m,i})$ by \eqref{eq:fromWjfin}. In other words,  
$\oe\{b_{m,i},b_{n,j}\}$ is an edge in $H$ which yields (B).  

\medskip

Consider next the case when  $W_0$ is infinite; let us outline the idea first
in the case when $k=2$. Bob will play one element sets at each step and aims
to build a one-way infinite square of a path following the lexicographical ordering on $\mathbb N \times \{0,1,2\}$. However, he picks the vertices in a different order, denoted by $\trianglelefteq$ later, which is demonstrated in Figure \ref{ordfig}.

\begin{figure}[H]
\centering
\psscalebox{0.9 0.9} 
{
\begin{pspicture}(0,-3.2)(14.19,2.7234793)
\psline[linecolor=black, linewidth=0.04, linestyle=dashed, dash=0.17638889cm 0.10583334cm, arrowsize=0.05291666666666667cm 2.0,arrowlength=1.4,arrowinset=0.0]{->}(2.4,-3.2)(12.0,-3.2)
\psline[linecolor=black, linewidth=0.04, linestyle=dashed, dash=0.17638889cm 0.10583334cm, arrowsize=0.05291666666666667cm 2.0,arrowlength=1.4,arrowinset=0.0]{->}(2.0,-2.8165207)(2.0,1.5834793)
\psline[linecolor=black, linewidth=0.04](12.0,1.5834793)(2.4,1.5834793)(2.4,0.5834793)(12.0,0.5834793)
\psline[linecolor=black, linewidth=0.04](12.0,-0.01652069)(2.4,-0.01652069)(2.4,-1.0165207)(12.0,-1.0165207)
\psline[linecolor=black, linewidth=0.04](12.0,-1.6165206)(2.4,-1.6165206)(2.4,-2.6165206)(12.0,-2.6165206)
\psline[linecolor=black, linewidth=0.04](2.2,-2.0165207)(1.6,-2.0165207)
\psline[linecolor=black, linewidth=0.04](2.2,-0.41652068)(1.6,-0.41652068)
\psline[linecolor=black, linewidth=0.04](2.2,1.1834793)(1.6,1.1834793)
\rput[bl](0,1.0){\Large $k=2$}
\rput[bl](12.6,-3.4){\Large $\mathbb N$}
\rput[bl](3.0,-2.2165208){\Large 1. }
\rput[bl](4.2,-2.2165208){\Large 2.}
\rput[bl](5.6,-2.2165208){\Large 4.}
\rput[bl](7.2,-2.2165208){\Large 7.}
\rput[bl](8.4,-2.2165208){\Large 10.}
\rput[bl](3.0,0.9834793){\Large 6.}
\rput[bl](4.2,0.9834793){\Large 9.}
\rput[bl](5.4,0.9834793){\Large 12.}

\rput[bl](3.0,-0.6165207){\Large 3.}
\rput[bl](4.2,-0.6165207){\Large 5.}
\rput[bl](5.6,-0.6165207){\Large 8.}
\rput[bl](12.5,0.9834793){\Large $\subseteq W_0$}
\rput[bl](12.5,-0.6165207){\Large $\subseteq W_1$}
\rput[bl](12.5,-2.4165206){\Large $\subseteq W_2$}
\rput[bl](7.2,-0.6165207){\Large 11.}
\psframe[linecolor=black, linewidth=0.04, linestyle=dashed, dash=0.17638889cm 0.10583334cm, dimen=outer, framearc=0.5](6.42,1.5034794)(5.2,-2.4765208)
\psframe[linecolor=black, linewidth=0.04, linestyle=dashed, dash=0.17638889cm 0.10583334cm, dimen=outer, framearc=0.5](8.1,-0.1)(6.92,-2.4765208)
\end{pspicture}
}

\caption{The two orderings.}%
\label{ordfig}
\end{figure}
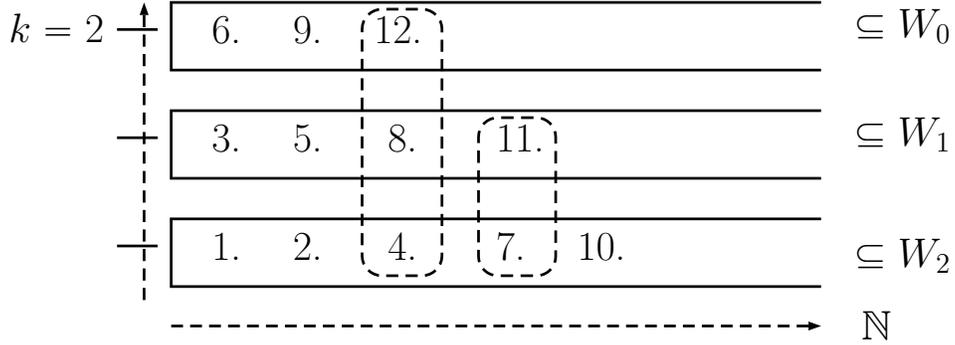

This way Bob makes sure that when he chooses the $12^{th}$ element he already 
picked its two successors (in the $7^{th}$ and $11^{th}$ plays) and two 
predecessors (in the $8^{th}$ and $4^{th}$ plays) in the lexicographical 
ordering, hence we can ensure the edge relations here.

Now, we define the strategy more precisely. In each round Bob will pick a single element $b_{n,j}$ for some $(n,j)\in \NN\times \{0,1,\dots, k\}$
such that $\big\{b_{n,j}:(n,j)\in \NN\times \{0,1,\dots, k\}\big\}$ will be the $k^{th}$ power of a path 
in the lexicographical order of $\NN\times \{0,1,\dots, k\}$.

As we said earlier, Bob will not choose the points $b_{n,j}$ in the lexicographical order of 
$\NN\times \{0,1,\dots, k\}$, i.e. typically the $((k+1)n+j)^{th}$ move of  Bob,  denoted by $B_{(k+1)n+j}$, is \emph{not} $\{b_{n,j}\}$.

To describe Bob's strategy we should define another order on $\NN\times \{0,1,\dots, k\}$
as follows:
\begin{align}
 (m,i)\trianglelefteq (n,j)\ \text{ iff }\  (m+i<n+j) \text{ or }  (m+i=n+j  \text{ and } i\le j).
\end{align}
Write $(m,i)\triangleleft (n,j)$ iff  $(m,i)\trianglelefteq (n,j)$ and $(m,i)\ne (n,j)$.
Clearly every $(n,j)$ has just finitely many $\triangleleft$-predecessors. Let $f(\ell)$ denote the 
$\ell^{th}$ element of $\NN\times \{0,1,\dots, k\}$ in the order $\triangleleft$.

Bob will choose $B_\ell=\{b_{f(\ell)}\}$ in the $\ell^{th}$ round as follows: if $f(\ell)=(n,j)$,
then 
\begin{enumerate}[(a)]
 \item if $j=0$ then 
\begin{align}\label{eq:fromw0}
b_{n,j}=\min \Big (W_0\setm \big (\bigcup_{s\le \ell}A_s\cup \bigcup_{t< \ell}B_t\big) \Big );  
\end{align}
 \item if $j>0$ then 
\begin{align}\label{eq:fromWj}
b_{n,j}\in W_j\cap N_H\big[\{b_{m,i}: (m,i)\triangleleft (n,j), i<j \}\big].  
\end{align}
\end{enumerate}

Bob can choose a suitable $b_{n,j}$ by \eqref{eq:winnig} as $\{b_{m,i}: (m,i)\triangleleft (n,j), i<j \}$ is a finite subset of $\bigcup_{i<j}W_{i}$.

We claim that 
\begin{enumerate}[(A)]
\item  $W_0\subseteq \bigcup\{A_n, B_n:n\in \mbb N\}$, and
 \item $P=\{b_{n,j}:n\in \mbb N, j\le k\}$ is the $k^{th}$-power of a path.
\end{enumerate}

(A) is clear because in (\ref{eq:fromw0}) we chose the minimal possible element.

Let $(m,i)\neq (n,j)\in \mathbb{N}\times \{0,\dots, k\}$. Then $b_{m,i}$ and  $b_{n,j}$ are  the  $((k+1)m+i)^{th}$ and  
$((k+1)n+j)^{th}$ elements, respectively, in the lexicographical order.  
Assume that $|((k+1)m+i)- ((k+1)n+j)|\le k$.
Then $i\ne j$ and $|m-n|\le 1$.

Without loss of generality, we can assume that $i<j$. Then $|m-n|\le 1$ implies $m+i\le n+j$ and hence 
$(m, i)\triangleleft (n,j)$. Since $i<j$ as well, $b_{n,j}\in N_H(b_{m,i})$ must hold by \eqref{eq:fromWj}. In other words,  
$\oe\{b_{m,i},b_{n,j}\}$ is an edge in $H$ which yields (B).  
\end{proof}
We arrive at one of our main results:

\begin{theorem}\label{tm:fonat1}
For all positive natural numbers $k,r$ and an $r$-edge coloring of $K_\NN$ the vertices can be partitioned into  $\leq r^{ (k-1) r+1}$ one-way infinite monochromatic $k^{th}$ powers of paths and a finite set.
\end{theorem}

\begin{proof}
The set of sequences of length $m$ (at most $m$, respectively) whose members are from a set $X$ is denoted by $X^m$ ($X^{\le m}$, respectively).

Recall that for each  $r$-edge coloring $c$ of $K_\NN$  Lemma \ref{uftrick} gives a partition
of the vertices, which we will denote by $d_c:\NN\to \{0, \dots, r-1 \}$, and a special color
$i_c < r$.
We define a set $A_s\subs \NN$ for each finite sequence $s\in \{0, \dots, r-1\}^{\le (k-1) r+1}$ by induction on
$|s|$ as follows:
\begin{itemize}
 \item let $A_\empt=\NN$,
 \item if $A_s$ is defined and finite then let
\begin{equation}
A_{s^\frown 0}=A_s\text{ and }A_{s^\frown i}=\empt\text{ for $1\le i<r$},
\end{equation}
\item if $A_s$ is defined and infinite then let
\begin{equation}
A_{s^\frown i}=\{u\in A_s:d_{c\restriction A_s}(u)=i\} \text{ for $i<r$}.
\end{equation}

\end{itemize}

Fix an arbitrary  $s\in  \{0, \dots, r-1\}^{(k-1) r+1}$ such that $A_s$ is infinite.  
Then there is a color $i_s<r$ and a $k$-element subset 
$H_s=\{h_1>h_2>\dots> h_{k}\}$ 
of $\{0, \dots, (k-1)r\}$ such that 
\begin{equation}
s(h_j)=i_s 
\end{equation}
for all $j=1, \dots, k$. 
Let  $W_0=A_s$  and  $W_j=A_{s\uhp h_j}$ for 
$j=1, \dots,  k$.
Note that the choice of $i_s$ ensures that
\begin{displaymath}
W_{j+1}\cap N_{G_s}[F]\text{ is infinite}
\end{displaymath}
   for each $j<\ekk$ and finite set 
$ F  \subset \bigcup_{i\le j}W_i$, 
where $G_{s}=(\NN, c^{-1}\{i_s\})$.
Thus, by Theorem
\ref{tm:winnig}, Bob has a winning strategy in the game
$\gm {G_{s}}{ A_s}\ekk$. 

Playing the games 
\begin{equation}
 \{\gm {G_{s}}{ A_s}\ekk:s\in \{0, \dots, r-1\}^{(k-1)r+1}\text{ and $A_s$ is infinite}\}
\end{equation}
simultaneously, that is, applying Lemma \ref{parallel} we can find at most $r^{(k-1)r+1}$ many $k^{th}$ powers
 of disjoint 
monochromatic paths
which cover $\NN$ apart from the finite set 
$\bigcup\{A_s:\text{$A_s$ is finite}\}$.  
\end{proof}

In the case of $k=r=2$, we have the following stronger result:

\begin{theorem}\label{tm:fonat2}
(1)  Given an edge coloring of  $K_{{\NN}}$
with $2$ colors, the vertices can be partitioned into 
$\leq$ 4 monochromatic path-squares  (that is, second powers of paths):
\begin{equation}
 \Covrel{K_{\NN}}{\mf{PathSquare}}{2,4}.
\end{equation}
(2)  The result above is sharp: there is  an edge coloring of  $K_{{\NN}}$
with $2$ colors such that the vertices cannot be covered by  3 monochromatic 
path-squares:
\begin{equation}
 \notCovrel{K_{\NN}}{\mf{PathSquare}}{2,3}.
\end{equation}
\end{theorem}

To prove Theorem \ref{tm:fonat2} we need some further preparation. First, in \cite[Corollary 1.10]{Po} Pokrovskiy proved the following:
{\em  Let $k,n \ge 1$ be natural numbers. Suppose that the edges of $K_n$ are colored with two colors.
 Then the vertices of
 $K_n$ can be covered with $k$ disjoint paths of color $1$ and  
a disjoint $k^{th}$ power of a path of color $0$.}

Second, we will apply the following

\begin{lemma}\label{lm:pa2sqr}
Assume that $P=v_0,v_1,\dots $ is a finite or one-way infinite path in a graph $G$ and there is $W\subs V(G)\setm P$ so that
\begin{equation}
\left(W\cap \nns{G}{\{v_i,v_{i+1},
v_{i+2},v_{i+3} \}} \right)  \text{ is infinite for all } v_i\in P.
\end{equation}
Let $\mc F$ be a countable family of infinite subsets of $W$.
Then $G$ contains a square of a path $R$ which covers $P$ while $R\setm P\subs W$, 
and 
$F\setminus R$ is infinite for all $F\in \mc F$. Moreover, if $P$ is finite
then $R$ can also be chosen to be finite.
\end{lemma}
  
\begin{proof}
Let $F_0, F_1,\dots$ be an enumeration of $\mc F$ in which each element shows up infinitely often. 

Pick distinct vertices $w_0,f_0, w_1,f_1,\dots$ from $W$
such that $$w_i\in \nns{G}{\{v_{2i},v_{2i+1},v_{2i+2},v_{2i+3} \}}\text{ and }
f_i\in F_i.$$

Then 
\begin{equation}
R=v_0, v_1, w_0, v_2, v_3, w_1, v_4,\dots, v_{2i}, v_{2i+1}, w_i,
 v_{2i+2}, v_{2i+3}, w_{i+1},\dots    
\end{equation}
is a   square of a path which covers $P$,   $R\setm P\subs W$,
and $\{f_n:n\in \NN, F_n=F\}\subseteq F\setm R$ for all $F\in \mc F$.

The last statement concerning the finiteness of $R$ is obvious.
\end{proof}

\begin{proof}[Proof of Theorem \ref{tm:fonat2}(1)]

Fix a coloring $c:\br {\NN};2;\to \{0, 1\}$ and let $G_i=({\NN}, c^{-1}\{i\})$
for $i<2$.

We will use the notation of Lemma \ref{uftrick}. Let $c_0=c$
and let
\begin{equation}
A_0=\{v\in {\NN}:d_{c_0}(v)=i_{c_0}\}\text{ and } B_0={\NN}\setm A_0. 
\end{equation}
Let $c_1=c_0\restriction B_0$ and provided $B_0$ is infinite we let
\begin{equation}
A_1=\{v\in B_0:d_{c_1}(v)=i_{c_1}\}\text{ and } B_1=B_0\setm A_1. 
\end{equation}

We can assume that $i_{c_0}=0$ without loss of generality.

\medskip
\noindent{\bf Case 1: }{\em  $B_0$ is finite.}

First, $G[B_0]$ can be written as the disjoint union of two paths $P_0$ and $P_1$ of color 1 and a
square of a path $Q$ of color 0 by the above mentioned  result of Pokrovskiy \cite[Corollary 1.10]{Po}.
Applying  Lemma \ref{lm:pa2sqr} for  $G=G_1$, $P=P_0$, $W=A_0$ and $\mc F=\empt$  it follows 
that there is a finite square of a path $R_0$ in color $1$ which covers $P_0$ and $R_0\setm P_0\subs A_0$. Applying  Lemma \ref{lm:pa2sqr} once more for  $G=G_1$, $P=P_1$, $W=A_0\setm R_0$ and $\mc F=\empt$  it follows 
that there is a finite square of a path $R_1$ in color $1$ which covers $P_1$,  and 
$R_1\setm P_1\subs  A_0\setm R_0$. 
Let $A_0'=A_0\setm (R_0\cup R_1)$.

Now, by Theorem \ref{tm:winnig}, Bob wins the game $\gm {G_0}{A_0'}2$ witnessed 
by the sequence $(A_0',A_0',A_0')$; 
thus $G[A_0']$ can be covered by a single square of a path $S$  of color 0 by Lemma \ref{parallel}. That is, $G$ can be covered by $4$  disjoint monochromatic squares of paths: $R_0$,
$R_1$, $Q$ and $S$.

\medskip
\noindent{\bf Case 2: }{\em  $B_0$ is infinite and $i_{c_1}=0$.}

Note that, by Theorem \ref{tm:winnig}, Bob wins the games
\begin{enumerate}[(i)]
 \item $\gm {G_0} {A_0} 2$ witnessed by $(A_0,A_0,A_0)$,
 \item $\gm {G_0} {A_1} 2$ witnessed by $(A_1,A_1,A_1)$,
 \item $\gm {G_1} {B_1} 2$ witnessed by $(B_1,A_1,A_0)$.
\end{enumerate}
Hence, the vertices can be partitioned into two squares of paths of color 0 
and a single square  
of a path of color 1 by Lemma \ref{parallel}. 

\medskip
\noindent{\bf Case 3: }{\em  $B_0$ is infinite and  $i_{c_1}=1$.}

Since we applied Lemma \ref{uftrick} twice to obtain $A_0$ and $B_0$, 
and $A_1$ and $B_1$,
and $B_1\subseteq B_0$
we know that 
\begin{enumerate}[(a)]
 \item Bob wins the game $\gm {G_0} {A_0} 2$ witnessed by $(A_0,A_0,A_0)$;
 \item Bob wins the game $\gm {G_1} {A_1} 2$ witnessed by $(A_1,A_1,A_1)$;
\item $N[F,1]\cap A_0$  is infinite for every finite set $F \subset B_1$;
\item  $N[F,0]\cap A_1$  is infinite for every finite set $F \subset B_1$;
\item \label{e} $N[F,0]\cap A_0$  is infinite for every finite set $F \subset A_0$;
\item \label{f} $N[F,1]\cap A_1$  is infinite for every finite set $F \subset A_1$.
\end{enumerate}

%

%

First,  partition $B_1$ into two paths 
$P_0$ 
 and $P_1$ 
 of color 0 and 1, respectively. Indeed, if $B_1$ is infinite this can be done by Theorem \ref{rado} and if $B_1$ is finite one considers two disjoint paths $P_0$ and $P_1$ in $B_1$ of color 0 and 1 with  $|P_0|+|P_1|$ maximal  (as outlined in a footnote in \cite{GeGy}); it is easily seen that $P_0\cup P_1$ must be $B_1$.

Now, our plan is to cover $P_0$ and $P_1$ with disjoint squares of paths $R_0$ and $R_1$ of color 0 and 1, respectively, such that
$R_0\setm P_0\subs A_1$, $R_1\setm P_1\subs A_0$ while
\begin{enumerate}[(a')]
 \item Bob wins the game $\gm {G_0} {A_0\setm R_1} 2$ witnessed by $(A_0\setm R_1,A_0\setm R_1,A_0\setm R_1)$,
 \item Bob wins the game $\gm {G_1} {A_1\setm R_0} 2$ witnessed by $(A_1\setm R_0,A_1\setm R_0,A_1\setm R_0)$.
\end{enumerate}
  
Let 
\begin{align}
\mc F_0=\{N[F,0]\cap A_0: F \subset A_0 \text{ finite}\},
\end{align}
and 
\begin{align}
\mc F_1=\{N[F,1]\cap A_1: F \subset A_1 \text{ finite}\},
\end{align}
and note that these families consist of infinite sets by \eqref{e} and \eqref{f} above.
Apply Lemma \ref{lm:pa2sqr} for $G=G_0$, $W=A_1$,  $P=P_0$ and $\mc F=\mc F_1$
to find a square of a path $R_0$ in $G_0$ which covers $P_0$, $R_0\setm P_0\subs A_1$ and $F\setm R_0$ is infinite for all
$F\in \mc F_1$, that is, 
\begin{equation}\label{win1}
N[F,1]\cap (A_1\setm R_0) \text{ is infinite for every finite set } F \subset A_1. 
\end{equation}

Apply Lemma \ref{lm:pa2sqr} once more for $G=G_1$, $W=A_0$,  $P=P_1$ and 
$\mc F=\mc F_0$
to find a square of a path $R_1$ in $G_1$ with $R_1\setm P_1\subs A_0$ which covers $P_1$ and 
$F\setm R_1$ is infinite for all
$F\in \mc F_0$, that is, 
\begin{equation}\label{win2}
N[F,0]\cap (A_0\setm R_1) \text{ is infinite for every finite set }  F \subset A_0. 
\end{equation}

Then, by Theorem \ref{tm:winnig}, \eqref{win2} yields (a'), and  \eqref{win1} yields (b'). 

Hence $(A_0\setm R_{1}) \cup (A_1\setm R_{0})$
can be partitioned into two monochromatic squares of paths by 
Lemma \ref{parallel} 
which in turn gives a partition of all the vertices into 4
monochromatic squares of paths.
%
 \end{proof}

\begin{proof}[Proof of Theorem \ref{tm:fonat2}(2)]
Fix a partition $(A,B,C,D)$ of ${\NN}$ such that 
 $A$ is infinite, $|B|=|C|=4$, and $|D|=1$.
Define the coloring $c: [\NN]^2\to \{0, 1\}$  as follows see Figure \ref{fig}:
\begin{equation}
c^{-1}\{1\}= \{\oe\{a, v\} :  a \in A,  v \in B\cup C\cup D \} \cup \br B;2; \cup \br C;2;.  
\end{equation}
\begin{figure}[H]
\centering
\begin{tikzpicture}[xscale=1, yscale=1]
\draw[thick]  (1,2)   -- (3,2) -- (3,4) -- (1,4) --  (1,2) -- (3,4) (3,2) -- (1,4)   ;
\draw[fill] (1,2)  circle (2pt) (3,2)  circle (2pt) (3,4)  circle (2pt) 
 (1,4)  circle (2pt);

\draw[thick]  (6,2)   -- (8,2) -- (8,4) -- (6,4) --  (6,2) -- (8,4) (8,2) -- (6,4)   ;
\draw[fill] (6,2)  circle (2pt) (8,2)  circle (2pt) (8,4)  circle (2pt) 
 (6,4)  circle (2pt);

 \draw[fill] (4.5,3)  circle (2pt) ;
\draw  (4.5,2.2) node[above]  {$D$} ;
\draw[thick]  (6,2) -- (8,2) -- (7,3) -- cycle;
\draw[dashed]  (0,0) rectangle (10,1)   ;
\draw  (5,0.1) node[above]  {$A$} ; 
\draw  (2,2.2) node[above]  {$B$} ; 
\draw  (7,2.2) node[above]  {$C$} ; 

\foreach \x in {0,1,...,44  }{ 
\draw[thick, xshift=0.2*\x cm] (0.5,0.8) -- (0.5,1.2) ;
}

\end{tikzpicture}
\caption{\label{fig} The example for Theorem \ref{tm:fonat2}(2)}
\end{figure}
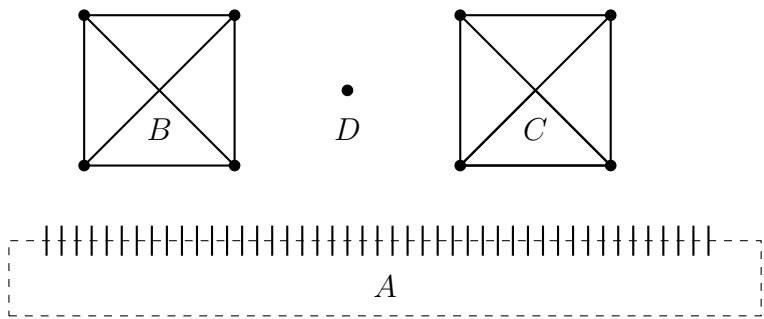

If $P$ is a monochromatic square of a path which intersects both $A$ and $B\cup C\cup D$,
then $P$ should be in color $1$, so $P\cap A$ should be finite.
Thus every partition of $K_{\NN}$ into monochromatic squares of paths
should contain an infinite  0-monochromatic square of a path $S\subs A$.  

It suffices to show now that $B\cup C\cup D$ cannot be covered by two monochromatic 
squares of paths. Let $D=\{d\}$.

First, if $P$ is a 1-monochromatic square of a path then $P'=P\cap (B\cup C\cup D)$
is a 1-monochromatic path. As two 1-monochromatic paths cannot cover
$B\cup C\cup D$, two 1-monochromatic squares of paths will not cover $B\cup C\cup D$ neither. 
 
Second, if $Q$ is a 0-monochromatic square of a path  which intersects 
$B\cup C\cup D$ then $Q\subs B\cup C\cup D$. Assume that $d\notin Q$ and let  $Q=x_1, x_2, \dots$. If $x_1\in B$ then $x_2\in C$ so $x_3$ does not exists because $Q$ is 
0-monochromatic square of a path. Hence $d\notin Q$ implies $|Q\cap B|\le 1$ and $|Q\cap C|\le 1$. If $d\in Q$, then cutting $Q$ into two by $d$ and using the observation above 
we yield that $|Q\cap B|\le 2$ and $|Q\cap C|\le 2$. In turn, two 0-monochromatic squares of  paths cannot cover $B\cup C\cup D$.

Finally using just one 0-monochromatic square of a  path  $Q$ 
we cannot cover $(B\cup C)\setm Q$ by a single 1-monochromatic square of a path because there is no 1-colored edge between $B\setm Q\ne \empt$ and $C\setm Q\ne \empt$.
\end{proof}

\section{Monochromatic path decompositions of $K_\oo$} \label{longpath}


The aim of this section is to extend Rado's Theorem \ref{rado} to 2-edge colored complete graphs of size $\omega_1$ (where $\omega_1$ is the smallest uncountable cardinal).

First, we need to extend certain definitions to the uncountable setting. 

\begin{definition}[Rado \cite{R}] We say that a graph $P=(V,E)$ is a
\emph{path} iff there is a well ordering $\prec$ on $V$ such that $$\{w\in
N_P(v):w\prec v\} \text{ is }\prec\text{-cofinal below } v$$ for all $v\in P$. 
\end{definition}

\begin{obs}\label{monpathobs} Suppose that $P=(V,E)$ is a graph and $\prec$ is a well ordering of $V$. Then the following are equivalent:
\begin{enumerate}
	\item $\prec$ witnesses that $P$ is a path,
	\item every $v,w\in V$ are connected by a $\prec$-monotone finite path in $P$.
\end{enumerate}
\end{obs}

In particular, each vertex is connected to its $\prec$-successor by an edge and so this general definition of a path coincides with the usual path notion for finite graphs.

The order type of $(V,\prec)$ above is called \emph{the order type of the path}.
We will say that a path $Q$ \emph{end extends} the path $P$ iff $P\subs Q$,
$\prec_Q\uhp P=\prec_P$ and $v\prec_Q w$ for all $v\in P, w\in Q\setm P$. If
$R$ and $S$ are two paths so that the first point of $S$ has $\prec_R$-cofinally many
neighbors in $R$ then $R\cup S$ is a path which end extends $R$ and we denote
this path by $R^\frown S$.

Let $K_\omg$ denote $(\omg,[\omg]^2)$,  i.e. the complete graph on $\omg$. Now we are ready to formulate the main result of this section. 

\begin{theorem}\label{2col}
\begin{equation}
 \Covrel{K_{\omega_1}}{\mf{Path},\mf{Path}}{2}.
\end{equation}    
That is, given any coloring of  the edges of  $K_{\oo}$
with $2$ colors, the vertices can be partitioned into two monochromatic   
paths of distinct colors. 
\end{theorem}

\subsection{Further preliminaries}

In the course of the proof we need more definitions.

\begin{definition} Let $\gr$ be a graph, $\kappa$ a cardinal and let $A\subs V$.
We say that $A$ is \emph{$\sat \kappa$} iff there are $\kappa$ many
disjoint finite paths between any two points of $A$. We say that $A$ is
\emph{$\kappa$-connected} iff there are $\kappa$ many disjoint finite
paths inside $A$ between any two points of $A$.
\end{definition}

We will apply this definition with $\kappa=\omega$ or $\omega_1$. We leave the (straightforward) proof of the next observation to the reader:

\begin{obs}\label{connobs} Let $\gr$ be a graph, $\kappa$ an infinite cardinal and let $A\subs V$. The following are equivalent:
\begin{enumerate}
	\item $A$ is $\sat \kappa$ ($\kappa$-connected),
	\item for every $v, w\in A$ and $F\subseteq V\setm\{v,w\}$ of size $<\kappa$ there is a finite path $P$ connecting $v$ and $w$ in $V\setm F$ (in $A\setm F$ respectively).
\end{enumerate}
\end{obs}


In the construction of a path longer than $\omega$, the difficulty lies in constructing the
 \emph{limit} elements. Definition \ref{traildef} will be crucial in overcoming this difficulty; the idea is first finding \emph{all limit vertices} of the path and then connecting these points appropriately.  

Recall  that a set $\mc V\subs [V]^\omega$ is a \emph{club} (closed and unbounded) iff
\begin{enumerate}
 \item $\bigcup\{V_n:n\in\omega\}\in \mc V$ for every increasing sequence $\{V_n:n\in\omega\}\subs \mc V$, and
\item for all $W\in [V]^\omega$ there is $U\in\mc V$ so that $W\subs U$.
\end{enumerate}

\noindent\emph{Remark.} An easy transfinite induction shows that every club on a set of size $\omg$ contains a club that is a well-ordered \emph{strictly increasing} family of the form $\{V_\alpha:\alpha<\omg\}$. Hence from now on we will tacitly assume that all clubs are of this form.

\begin{definition}\label{traildef}Suppose that $G=(V,E)$ is a graph with $|V| = \omg$. We say that $A\subset V$ is a
 \emph{trail} iff there is a club $\{V_\alpha:\alpha<\omg\} \subset [V]^\omega$ so that 
for all $\alpha<\omg$ there is $v_\alpha\in A\setm V_\alpha$ such that for all $\alpha'<\alpha$
\begin{equation}
 N_G(v_\alpha)\cap (V_\alpha\setm V_{\alpha'})\cap A \text{ is infinite}.
\end{equation}
\end{definition}

\begin{figure}[H]
\centering
\psscalebox{0.8 0.7} 
{
\begin{pspicture}(0,-2.9994395)(14.44,2.9994395)
\psline[linecolor=black, linewidth=0.04](12.82,-0.058148347)(0.02,-0.058148347)(0.02,-1.6581483)(12.82,-1.6581483)
\psbezier[linecolor=black, linewidth=0.04](0.42,1.9839051)(1.3279594,2.4029632)(5.62,1.583905)(5.62,-0.41609496)(5.62,-2.416095)(1.3279594,-3.2351532)(0.42,-2.8160949)
\psbezier[linecolor=black, linewidth=0.04](6.42,2.383905)(7.3686833,2.7001328)(11.62,1.183905)(11.22,-0.816095)(10.82,-2.8160949)(6.05205,-3.3707952)(5.22,-2.8160949)
\psdots[linecolor=black, dotsize=0.3](12.42,-0.85814834)
\psdots[linecolor=black, dotsize=0.2](2.42,-0.85814834)
\psdots[linecolor=black, dotsize=0.2](6.42,-0.85814834)
\psdots[linecolor=black, dotsize=0.2](7.22,-0.85814834)
\psdots[linecolor=black, dotsize=0.14](9.62,-0.85814834)
\psdots[linecolor=black, dotsize=0.09](10.02,-0.85814834)
\psdots[linecolor=black, dotsize=0.07](10.42,-0.85814834)
\psbezier[linecolor=black, linewidth=0.04, linestyle=dashed, dash=0.17638889cm 0.10583334cm](12.42,-0.85814834)(12.82,-0.45814836)(10.02,1.5418516)(7.22,1.5418516)(4.42,1.5418516)(2.82,-0.058148347)(2.42,-0.85814834)
\psbezier[linecolor=black, linewidth=0.04, linestyle=dashed, dash=0.17638889cm 0.10583334cm](12.42,-0.85814834)(12.02,-0.058148347)(10.02,1.1418517)(8.42,1.1418517)(6.82,1.1418517)(6.42,-0.45814836)(6.42,-0.85814834)
\psbezier[linecolor=black, linewidth=0.04, linestyle=dashed, dash=0.17638889cm 0.10583334cm](12.42,-0.85814834)(12.02,-0.45814836)(10.42,0.34185165)(9.22,0.34185165)(8.02,0.34185165)(7.22,-0.45814836)(7.22,-0.85814834)
\psbezier[linecolor=black, linewidth=0.04, linestyle=dashed, dash=0.17638889cm 0.10583334cm](12.42,-0.85814834)(12.02,-0.45814836)(10.82,-0.45814836)(10.42,-0.45814836)(10.02,-0.45814836)(9.62,-0.85814834)(9.62,-0.85814834)
\rput[bl](13.22,-0.85814834){\large $v_\alpha$}
\rput[bl](8.5,2.2){\Large $V_\alpha$}
\rput[bl](1.62,2.2){\Large $V_{\alpha'}$}
\rput[bl](0.02,0.7){\Large $A$}
\end{pspicture}
}\caption{Trails.}
\label{trailfig}
\end{figure}
\vspace{0.2 cm}

An important example of a path is the graph $\halff \omg$ i.e. $(\omg\times 2,E)$ where
\begin{displaymath}
E=\big\{\ \{({\alpha},0),({\beta},1)\}: {\alpha}\le {\beta}<{\omg}\ \big\}. 
\end{displaymath}

 $\halff \omg$ is a bipartite graph and we call the set of vertices in
 $\halff\omg$ with degree $\omg$, (that is, $\omega_1 \times \{0\}$) \emph{the
   main class of $\halff \omg$}.

\begin{obs}\label{trailobs}The main class of $\half$ is $\omega_1$-linked
  and is a trail (indeed, let $V_\alpha = \omega\alpha \times 2$). 

If $G$ is any graph and $A\subs V(G)$ is a 
trail then
 \begin{enumerate} 
\item $C$ is a trail for any $A\subs C\subs V(G)$,
\item $C\setm B$ is a trail for any $B\in [C]^\omega$,
\item if $\{W_\alpha:\alpha<\omg\} \subset \{V_\alpha:\alpha<\omg\}$ are clubs and $\{V_\alpha:\alpha<\omg\}$ witnesses that $A$ is a trail then so does $\{W_\alpha:\alpha<\omg\}$. 
 \end{enumerate}
\end{obs}

We will make use of the following lemma regularly but the reader should feel free to skip the proof when first working through this section.

\begin{lemma} \label{ellemma} Let $\gr$ be a graph with $|V|=\omg$, and let $A\subseteq V$ be uncountable. Then there is a club $\{V_\alpha:\alpha<\omg\}$ of $V$ such that 
\begin{enumerate}
	\item $V_\alpha$ is an initial segment of $\omg$ and if $\xi\in V_\alpha$ then $\xi+1\in V_\alpha$ as well,
	\item if $A$ is $\sat \omg$ ($\omg$-connected) then $A\cap (V_{\alpha+1}\setm V_{\alpha})$ is infinite and $\sat \omega$ ($\omega$-connected) in $V_{\alpha+1}\setm V_{\alpha}$,
	\item if $A$ is a trail then $\{V_\alpha:\alpha<\omg\}$ witnesses this and, using the notation of Definition \ref{traildef}, the node $v_\alpha$ can be chosen in $V_{\alpha+1}\setm V_{\alpha}$.
\end{enumerate}
\end{lemma}
\begin{proof} Let $\mc M=\{M_\alpha:0<\alpha<\omg\}$ be an $\in$-chain of countable elementary submodels of $H(\omega_2)$ such that $G,A,\prec \in M_1$, $V\subseteq \bigcup \mc M$ and let $M_0=\emptyset$. Let $V_\alpha=V\cap M_\alpha$ for $\alpha<\omg$. We claim that $\mc V=\{V_\alpha:\alpha<\omg\}$ is a club which satisfies the above conditions.

First, $\mc V$ is a club as $\mc M$ is a continuous chain and $V\subseteq \bigcup \mc M$. Condition (1) is satisfied by elementarity.

Now, suppose that $A$ is $\sat \omg$ and fix $\alpha<\omg$. Also, fix $v,w\in
V_{\alpha+1}\setm V_{\alpha}$ and a finite set $F\subs V_{\alpha+1}\setm
V_{\alpha}$. We prove that there is a path form $v$ to $w$ in
$V_{\alpha+1}\setm (V_{\alpha}\cup F)$; this implies that $A\cap
(V_{\alpha+1}\setm V_{\alpha})$ is $\sat \omega$ by Observation
\ref{connobs}. As $A$ is $\sat \omg$ we have that $$H(\omega_2)\models \text{
  there is a finite path from } v \text{ to }w \text{ in } V\setm(V_\alpha\cup
F).$$ Hence, by elementarity of $M_{\alpha+1}$ and by $V_\alpha, F,v,w\in
M_{\alpha+1}$ we have $$M_{\alpha+1}\models \text{ there is a finite path from
} v \text{ to }w \text{ in } V\setm(V_\alpha\cup F).$$ We choose any such path
$P$ in $M_{\alpha+1}$ and so we have $P\subseteq
V_{\alpha+1}\setm(V_\alpha\cup F)$ as desired. The case when $A$ is
$\omg$-connected is completely analogous.

Finally, suppose that $A$ is a trail. By elementarity, since $A \in
M_1$, there is a club
$\mc W=\{W_\alpha:\alpha<\omg\}\in M_1$ which witnesses that $A$ is a
trail. First, it is easy to see that $\mc V\subseteq \mc W$ and in particular,
$\mc V$ witnesses that $A$ is a trail. Second, the node $v_{\alpha}\in V\setm
V_{\alpha}$ can be selected in $V_{\alpha+1}$ as $V_\alpha \in M_{\alpha+1}$
and
\begin{equation}
M_{\alpha+1}\models \text{ there is } v\in V\setm V_\alpha \text{ such that }|N_G(v)\cap (V_\alpha\setm V_{\alpha'})\cap A|=\omega \text{ for all } \alpha'<\alpha.
\end{equation}
This finishes the proof of the lemma.
\end{proof}

Finally, we state the obvious extension of Lemma \ref{uftrick}.

\begin{lemma}\label{uftrick2} Given any edge coloring $c: [\kappa]^2 \to \{0, 
\dots, r-1 \}$ of the complete graph on $\kappa$ (where $\kappa\geq\omega$), 
there is a function  $d_c: \kappa \to \{0, \dots, r-1 \}$ and a color $i(c) < r$
so that  the sets 
$V_i=d_c^{-1}\{i\}$ satisfy:
$$|N[F,i]\cap V_{i(c)}|=\kappa \text{ for all } i < r \text{ and finite set } 
F\subs V_i.$$

In particular, $V_i$ is $\sat \kappa$ in color $i$ for all $i < r$ and 
$V_{i(c)}$
is $\kappa$-connected in color $i(c)$. 
\end{lemma}

\begin{proof}
Repeat tho proof of Lemma \ref{uftrick} but choose the ultrafilter $U$ on $\kappa$ to be uniform, that is, $|H| = \kappa$ for every $H \in U$. 
\end{proof}

\subsection{Towards the proof of Theorem \ref{2col}}

The following two lemmas express the connection between trails, $\sat \omg$ sets and paths:

\begin{lemma}\label{pathlemma}
 Every path of order type $\omg$ is a trail and contains an uncountable $\sat \omg$ subset.
\end{lemma}
\begin{proof}  Suppose that $P$ is a path of order type $\omg$ witnessed by the well ordering $\prec$. Now, by Lemma \ref{ellemma}, there is a club $\{V_\alpha:\alpha<\omg\}$ of vertices of $P$ such that $V_\alpha$ is a $\prec$-initial segment, $v\in V_\alpha$ implies that the $\prec$-successor of $v$ is also in $V_\alpha$ and $V_\alpha\subs V_\beta$ for all $\alpha<\beta<\omg$. Let $v_\alpha$ denote the $\prec$-minimal element of $V\setm V_\alpha$.
In order to prove that $P$ is a trail it suffices to show that 
\begin{claim}
 $N_P(v_\alpha)\cap (V_\alpha\setm V_{\alpha'})$ is infinite for all
$\alpha'<\alpha<\omg$.
\end{claim}
\begin{proof} First, note that $v_\alpha$ is a $\prec$-limit. Fix $\alpha'<\alpha$. $V_{\alpha'}$ is an initial segment of the path $P$ and has minimal bound $v_{\alpha'}$. Note that $v_{\alpha'}\prec v_\alpha$. By the definition of a path, the set $\{w\in N_P(v_\alpha):v_{\alpha'}\prec w\prec v_\alpha\}$ is infinite and it is clearly a subset of $N_P(v_\alpha)\cap V_\alpha\setm V_{\alpha'}$ by the choice of $v_{\alpha'}$ and $v_\alpha$. 
\end{proof}

 Second, we prove

\begin{claim}
 The set $A=\{v\in V(P): |N_P(v)|=\omega_1\}$ is uncountable and $\sat \omg$.
\end{claim}
\begin{proof}
First, it suffices to show that there is a single vertex $v$ with uncountable
degree in $P$ as
 every end segment of $P$ is also a path of order type $\omg$. Let $\{p_\alpha:\alpha<\omg\}$ enumerate $P$ according to the path well order $\prec$. Now, for every
limit $\alpha<\omg$
 there is $\mu_\alpha<\alpha$ so that $\{p_\alpha,p_{\mu_\alpha}\}\in E(P)$;
Fodor's pressing down 
lemma gives a stationary set $S\subs \omg$ and $\mu\in \omg$ so that
$\{p_\alpha,p_{\mu}\}\in E(P)$
 if $\alpha\in S$, that is, the degree of $p_\mu$ in $P$ is uncountable.

Now take any two distinct vertices, $v$ and $w$, in $A$ and fix an
arbitrary countable set $F\subs V(P)\setm\{v,w\}$. We will find a finite path from $v$ to $w$ in $V(P)\setm F$. There is $v'\in N_P(v)$ and $w'\in
N_P(w)$ so that both $v'$
 and $w'$ are $\prec$-above all elements of $F$ as $v,w\in A$ and $|F|\leq \omega$.
Now, there is a finite 
$\prec$-monotone path $Q$ between $v'$ and $w'$ by Observation
\ref{monpathobs}; $Q$ must avoid $F$ and hence the path $(v)\smf Q\smf (w)$
connects $v$ and $w$ in $V(P)\setm F$. By Observation \ref{connobs}, $A$ must
be $\sat {\omega_1}$.
\end{proof}

\end{proof}

Now, we show that the converse of Lemma \ref{pathlemma} is true as well:

\begin{lemma}\label{trail} Suppose that $\gr$ is a graph with $|V| = \omg$.
 If $V$ is an $\omg$-connected trail then $G$ is a path.
\end{lemma}
\begin{proof}Fix a club $\{V_\alpha:\alpha<\omg\}$ as in Lemma \ref{ellemma} and pick nodes $v_\alpha\in V_{\alpha+1}\setm V_\alpha$ showing that $V$ is a trail.

It suffice to construct sets $P_\alpha\subs V$ and orderings 
$\prec_\alpha$ for $\alpha<\omg$ 
so that
\begin{enumerate}[(i)]
\item $(P_\alpha,\prec_\alpha)$ is a path with last point $v_\alpha$,
\item $P_\alpha=V_\alpha\cup \{v_\alpha\}$,
\item $P_\beta$ end extends $P_\alpha$ for $\alpha<\beta<\omg$.
\end{enumerate}
Indeed, the ordering $\bigcup\{\prec_\alpha:\alpha<\omg\}$ on $V$ will witness that $G$ is a path.

First, we set $P_0=\{v_0\}$. Next, apply Corollary \ref{trivpath} to find a 
path $R$ of order type $\omega$ on vertices $V_1$ with first point $v_0$; this can be done as $V_1$ is $\omega$-connected.  We let $P_1=R^\frown (v_1)$ and note that $P_1$ is a path as the infinite set $N_G(v_1)\cap V_1$ is cofinal in $R$ and hence below $v_1$.

In general, suppose that we have constructed $P_\alpha$ for $\alpha<\beta$ as above. 
If $\beta$ is a 
limit then let $P_{<\beta}=\bigcup\{P_\alpha:\alpha<\beta\}$; note that
 $P_{<\beta}=V_\beta$ is a path.
 It suffices to prove 

\begin{obs}
 $P_\beta=P_{<\beta}\hspace{0.01 cm}^\frown (v_\beta)$ is a path. 
\end{obs}

Indeed, we know that $N_G(v_\beta)\cap (V_\beta\setm V_{\alpha})$ is infinite for all $\alpha<\beta$ by 
the definition of $v_\beta$.

If $\beta=\alpha+1$ then we apply Lemma \ref{trivpath} to find a path $R$ of order type $\omega$ on 
 vertices $V_{\alpha+1}\setm V_\alpha$ with first point $v_\alpha$; see Figure \ref{extfig}.

\begin{figure}[H]
\centering

\psscalebox{0.9 0.9} 
{
\begin{pspicture}(-1.5,-6.278789)(16.071987,-1)
\psbezier[linecolor=black, linewidth=0.04](0.0119877625,-2.121211)(0.9311328,-1.7272918)(5.6119876,-2.5212111)(5.6119876,-4.121211)(5.6119876,-5.721211)(0.9735117,-6.395932)(0.0119877625,-6.121211)
\psbezier[linecolor=black, linewidth=0.04](6.811988,-2.121211)(7.731133,-1.7272918)(12.411987,-2.5212111)(12.411987,-4.121211)(12.411987,-5.721211)(7.773512,-6.395932)(6.811988,-6.121211)
\psline[linecolor=black, linewidth=0.04, arrowsize=0.05291666666666667cm 2.0,arrowlength=1.4,arrowinset=0.0]{->}(0.0119877625,-4.9212112)(1.6119877,-3.7212112)(3.6119878,-4.9212112)(4.811988,-4.121211)
\psdots[linecolor=black, dotsize=0.2](0.7319878,-4.3812113)
\psdots[linecolor=black, dotsize=0.2](1.6319878,-3.7212112)
\psdots[linecolor=black, dotsize=0.2](2.5319877,-4.2612114)
\psdots[linecolor=black, dotsize=0.2](3.2319877,-4.681211)
\psdots[linecolor=black, dotsize=0.3](6.391988,-4.081211)
\psdots[linecolor=black, dotsize=0.3](13.2119875,-4.081211)
\psdots[linecolor=black, dotsize=0.2](7.9919877,-4.8812113)
\psdots[linecolor=black, dotsize=0.2](9.571988,-4.081211)
\psdots[linecolor=black, dotsize=0.2](10.811988,-4.481211)
\psbezier[linecolor=black, linewidth=0.04, linestyle=dashed, dash=0.17638889cm 0.10583334cm](13.2119875,-4.021211)(13.318035,-3.3353179)(11.841863,-2.4992669)(10.331987,-2.621211)(8.822113,-2.7431555)(7.85262,-3.9640183)(7.911988,-4.7812114)
\psbezier[linecolor=black, linewidth=0.04, linestyle=dashed, dash=0.17638889cm 0.10583334cm](13.171988,-4.081211)(12.855649,-3.3043356)(11.591743,-2.9633052)(11.091988,-3.0212111)(10.592232,-3.079117)(9.761014,-3.247157)(9.551988,-4.001211)
\psbezier[linecolor=black, linewidth=0.04, linestyle=dashed, dash=0.17638889cm 0.10583334cm](13.171988,-4.041211)(12.87794,-3.5334666)(11.95948,-3.6035442)(11.551988,-3.7012112)(11.144495,-3.798878)(10.77331,-3.9758096)(10.791987,-4.4212112)
\psline[linecolor=black, linewidth=0.04](6.391988,-4.081211)(7.9719877,-4.8412113)
\psline[linecolor=black, linewidth=0.04, linestyle=dashed, dash=0.17638889cm 0.10583334cm, arrowsize=0.05291666666666667cm 2.0,arrowlength=1.4,arrowinset=0.0]{->}(8.031988,-4.8412113)(9.551988,-4.061211)(10.451988,-4.3812113)
\psbezier[linecolor=black, linewidth=0.04](6.431988,-4.041211)(5.839779,-3.351742)(5.2359557,-3.4028668)(4.711988,-3.5412111)(4.1880198,-3.6795554)(3.411738,-3.862069)(3.2319877,-4.621211)
\psbezier[linecolor=black, linewidth=0.04](6.4719877,-4.061211)(5.99486,-3.1115983)(4.6495295,-2.6312864)(3.5119877,-2.7212112)(2.3744462,-2.811136)(1.7609899,-3.0888)(1.6319878,-3.7212112)
\rput[bl](6.291988,-4.981211){\large $v_\alpha$}
\rput[bl](13.231988,-4.8412113){\large $v_{\alpha+1}$}
\rput[bl](2.4919877,-1.6412112){\Large $V_\alpha$}
\rput[bl](9.951988,-1.6412112){\Large $V_{\alpha+1}$}
\rput[bl](1.4719877,-5.2){\Large $P_\alpha$}
\rput[bl](9.471988,-5.2){\large $R$}
\end{pspicture}}
\caption{Extending $P_\alpha$ to $P_{\alpha+1}$.}
\label{extfig}
\end{figure}
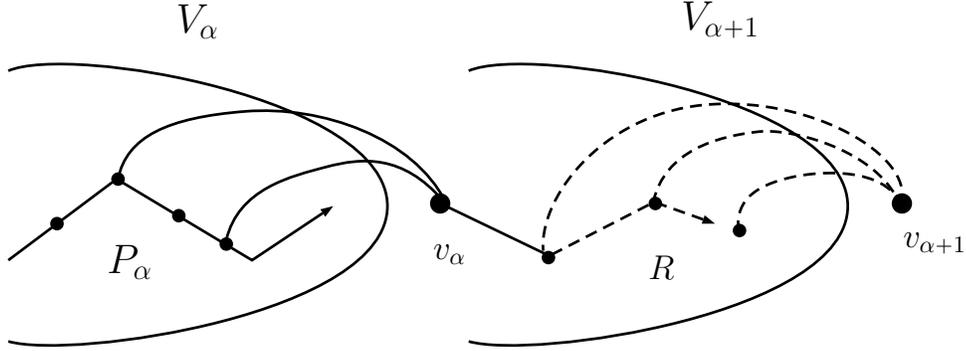

\vspace{0.2 cm}

 We let $P_{\alpha+1}=P_\alpha\hspace{0.01 cm}^\frown R\hspace{0.01 cm}^\frown(v_{\alpha+1})$. Note that $P_{\alpha+1}$ is a path as the infinite set $N_G(v_{\alpha+1})\cap (V_{\alpha+1}\setm V_\alpha)$  is cofinal in $R$ and hence below $v_{\alpha+1}$.
\end{proof}

We note that it is proved very similarly that if a set of vertices $A\subseteq V$ is an $\sat \omg$ trail then $A$ can be \emph{covered} by a path of order type $\omg$.

Finally, before the proof of Theorem \ref{2col}, we prove a simple result about finding trails in 2-edge colored copies of $K_\omg$.

\begin{lemma}\label{dich} Suppose that $c$ is a 2-edge coloring of $K_\omg$ and $A\subs \omg$. Then either $A$ is a trail in color 0 or we can find a copy of $K_\omg$ in color 1 (inside $A$). 

\end{lemma}

 \begin{proof} Suppose that $A$ is not a trail in color 0. Let
   $V_\alpha=\omega\alpha \subs \omg$ (regarded as a set of vertices) for
   $\alpha<\omg$. Let $v_\alpha=\min (A\setm
   V_\alpha)$ for $\alpha<\omg$. Let $$X=\{\alpha<\omg: |N(v_\alpha,0) \cap
   (V_\alpha\setm V_{\alpha'}) \cap A|=\omega \text{ for all
   }\alpha'<\alpha\}.$$ If there is a club $C$ in $X$ then
   $\{V_\alpha:\alpha\in C\}$ witnesses that $A$ is a trail in color 0. Hence,
   as $X$ cannot contain a club, there is a stationary set $S$ in $\omg\setm
   X$. We can suppose, by shrinking $S$ to a smaller stationary set, that
   $\omega\alpha=\alpha$ for all $\alpha\in S$. Now for every
$\alpha \in S$ there is $\nu_\alpha<\alpha$ and finite $F_\alpha\subs V_\alpha$ such that 
$$F_\alpha=N(v_\alpha,0)\cap (V_\alpha\setm V_{\nu_\alpha}) \cap A.$$ By Fodor's
pressing down lemma we can find stationary $T\subs S$, $\nu<\omg$ and finite
set $F$ such that $$F=N(v_\alpha,0)\cap (V_\alpha\setm V_{\nu}) \cap A$$ for all
$\alpha\in T$. It is clear now that $B=\{v_\alpha:\alpha\in T\}\setm
(V_\nu\cup F)$ is an uncountable subset of $A$ and $c\uhp [B]^2\equiv 1$.
\end{proof}

Now, we are ready to prove the main result of this section:

\begin{proof}[Proof of Theorem \ref{2col}] Fix an edge coloring $r:[\omg]^2\to 2$ of the complete graph $K_{\omg}=(\omg,[\omg]^2)$. We distinguish two cases as follows:

 \textbf{Case 1:} There is a monochromatic copy $H_0$ of $\half$. 
 
We can suppose that $H_0$ is 
0-monochromatic by symmetry and let $A$ denote the main class of $H_0$. As $A$ is $\sat \omg$ in color $0$, we can extend $A$ to a \emph{maximal} subset $C \subseteq \omega_1$ that is $\sat \omg$ in color $0$. Note that, by the maximality of $C$,
\begin{equation}\label{eq:NycapC}
|N(v,0)\cap C|\leq \omega \text{ for all } v\in \omg\setm C,
\end{equation}
and in particular $\omg\setm C$ is $\sat \omg$ in color 1.

\textbf{Case 1A:} $\omg\setm C$ is countable.

Find a path $P^1$ in color 1 and of order type $\leq \omega$ which covers $\omg\setm C$ (see Corollary \ref{trivpath}). 
\begin{claim} $C\setm P^1$ is an $\omg$-connected trail in color 0.
\end{claim} 

Indeed, if $v,w\in C\setm P^1$ and $F\subs V$ is countable then there is a path $P$ of color 0 from $v$ to $w$ which avoids $F\cup P^1$ as $C$ is $\sat \omg$ in color $0$; in particular, $P$ is also a subset of $C\setm P^1$ and in turn $C\setm P^1$ is $\omg$-connected in color $0$. By Observation \ref{trailobs} $C$ is a trail in color 0 witnessed by the copy of
$\half$, and hence , using Observation \ref{trailobs} again, $C\setm P^1$ remains a trail as well. This finishes the proof of the claim.

Hence, by Lemma \ref{trail}, $C\setm P^1$ is a path in color 0 which finishes the proof of Theorem \ref{2col} in Case 1A.

\textbf{Case 1B:} $\omg\setm C$ is uncountable.

\begin{claim}\label{halfcover} $\omg\setm C$ is covered by a copy of $\half$ in color 1 with main class $\omg\setm C$.
\end{claim}
 \begin{proof}Note that $|N[X,1]\cap C|=\omg$ for all $X\in [\omg\setm C]^\omega$ by (\ref{eq:NycapC}). Enumerate $\omg\setm C$ as $\{x_\alpha:\alpha<\omg\}$ and inductively select $$y_\beta\in N[X_\beta,1]\cap C\setm \{y_\alpha:\alpha<\beta\}$$ for $\beta<\omg$ where $X_\beta=\{x_\alpha:\alpha\leq\beta\}$. Now $(\omg\setm C) \cup\{y_\alpha:\alpha<\omg\}$ is the desired copy of $\half$ in color 1.
\end{proof}

Let $H_1$ denote this copy of $\half$. Our goal is to mimic the proof of Lemma \ref{trail}  and, using $H_0$ and $H_1$, simultaneously construct two disjoint monochromatic paths (one in color 0
and one in color 1) which cover $V$.

Using Lemma \ref{ellemma} twice, the observation that the intersection of two clubs is itself a club, and also Observation \ref{trailobs}, we can fix a club $\{V_\alpha:\alpha<\omg\}$ in $\omg$ such that $C\cap (V_{\alpha+1}\setm V_\alpha)$ and $(\omg\setm C)\cap (V_{\alpha+1}\setm V_\alpha)$ are $\sat \omega$ in color 0 and 1, respectively, inside $V_{\alpha+1}\setm V_\alpha$ for all $\alpha<\omg$. Furthermore, we can suppose that $V_\alpha$ intersects $H_0$ and $H_1$ in initial segments of their respective $\half$ orderings for each $\alpha<\omg$. 

Now, we inductively construct disjoint sets $P^0_\alpha,P^1_\alpha$ and well orderings  $\prec^0_\alpha,\prec^1_\alpha$ such that
\begin{enumerate}[(i)]
\item $(P^i_\alpha,\prec^i_\alpha)$ is a path in color $i$ of order type $\omega$ for $i<2$,
\item $P^i_\beta$ end extends $P^i_\alpha$ for all $\alpha<\beta$ and $i<2$,
\item $A\cap P^0_\alpha$ is cofinal in $P^0_\alpha$ and $(\omg\setm C)\cap P^1_{\alpha}$ is cofinal in
$P^1_\alpha$,
\item $P^0_\alpha\cup P^1_\alpha=V_\alpha$
\end{enumerate}
for all $\alpha<\omg$.

First, apply Lemma \ref{simpath} to $G[V_1]$ and the sets $C_0 = C\cap V_1$, $C_1 = (\omg\setm C) \cap V_1$, $A_0 = A\cap V_1$ and $A_1 = \emptyset$ to obtain two disjoint paths covering $V_1$: $P^0_1$ in color 0 and $P^1_1$ in color 1, both of order type $\omega$. Since  $P^0_1$ and $P^1_1$ are of order type $\omega$, a subset of such a path is cofinal iff it is infinite. Lemma \ref{simpath} makes sure that $A\cap P^0_1$ as well as $(\omg\setm C)\cap P^1_1$ are infinite (note that $A_0$ is infinite), hence (iii) holds.

Suppose we have constructed $P^0_{\alpha},P^1_{\alpha}$ as above for $\alpha<\beta$. Note
 that $P^0_{<\beta}=\bigcup\{P_{\alpha}:\alpha<\beta\}$ is a path in color 0, 
$P^1_{<\beta}=\bigcup\{P^1_{\alpha}:\alpha<\beta\}$ is a path in color 1 and $A\cap P^0_{<\beta}$ is cofinal in $P^0_{<\beta}$ while $\omg\setm C$ is cofinal in $P^1_{<\beta}$. Thus if $\beta$ is limit we are done. Suppose that $\beta=\alpha+1$, i.e. $P^i_{<\beta}=P^i_\alpha$ for $i<2$. 

\begin{claim}\label{cl:back} 

\begin{enumerate}[(a)]
	\item There are $v^0_\alpha,w^0_\alpha\in V_{\alpha+1}\setm V_\alpha$ such that $A\cap V_\alpha\subseteq N(v^0_\alpha,0)$, $w^0_\alpha\in A$ and $c(v^0_\alpha,w^0_\alpha)=0$. 
	\item There are $v^1_\alpha,w_\alpha^1\in V_{\alpha+1}\setm (V_\alpha\cup\{v^0_\alpha,w^0_\alpha\})$ such that $(\omg\setm C )\cap V_\alpha\subseteq N(v^1_\alpha,1)$, $w^1_\alpha\in \omg\setm C$ and $c(v^1_\alpha,w^1_\alpha)=1$.
\end{enumerate}
\end{claim}
 \begin{proof}
(a) We know that $V_\alpha$ intersects $H_0$ in an initial segment and hence any element $v^0_\alpha$ from $(V(H_0)\setm A)\cap (V_{\alpha+1}\setm V_\alpha)$ will satisfy $A\cap V_\alpha\subseteq N(v^0_\alpha,0)$. We can now select $w^0_\alpha\in A \cap(V_{\alpha+1}\setm V_\alpha)$ such that $c(v^0_\alpha,w^0_\alpha)=0$.

The proof of (b) is completely analogous to (a).
\end{proof}

\begin{figure}[H]%
\centering
\psscalebox{0.9 0.9} 
{
\begin{pspicture}(0,-2.62)(11.24,2.62)
\psbezier[linecolor=black, linewidth=0.04](0.02,1.8)(0.82,1.8)(5.62,1.36)(5.62,-0.4)(5.62,-2.16)(0.82,-2.6)(0.02,-2.6)
\psbezier[linecolor=black, linewidth=0.04](7.02,1.8)(7.705714,1.8)(11.82,1.36)(11.82,-0.4)(11.82,-2.16)(7.705714,-2.6)(7.02,-2.6)
\rput[bl](2.02,2.2){\Large $V_\alpha$}
\rput[bl](9.42,2.0){\Large $V_{\alpha+1}$}
\psline[linecolor=black, linewidth=0.04](7.62,0.2)(0.02,0.2)(0.02,-1.0)(7.62,-1.0)
\rput[bl](0.02,0.6){\Large $A$}
\psline[linecolor=black, linewidth=0.04, arrowsize=0.05291666666666667cm 2.0,arrowlength=1.4,arrowinset=0.0]{->}(0.02,-1.6)(0.82,-0.4)(1.82,-1.6)(2.82,-0.4)(3.62,-1.6)(5.02,-0.4)
\psdots[linecolor=black, dotsize=0.2](6.42,0.8)
\psdots[linecolor=black, dotsize=0.2](7.42,-0.6)
\psline[linecolor=black, linewidth=0.04, linestyle=dashed, dash=0.17638889cm 0.10583334cm, arrowsize=0.05291666666666667cm 2.0,arrowlength=1.4,arrowinset=0.0]{->}(6.42,0.8)(7.22,-0.4)
\psdots[linecolor=black, dotsize=0.2](0.82,-0.4)
\psdots[linecolor=black, dotsize=0.2](2.82,-0.4)
\psdots[linecolor=black, dotsize=0.2](1.82,-1.6)
\psdots[linecolor=black, dotsize=0.2](3.62,-1.6)
\psbezier[linecolor=black, linewidth=0.04, linestyle=dashed, dash=0.17638889cm 0.10583334cm](6.42,0.8)(5.62,1.6)(3.5978024,1.4095291)(2.62,1.2)(1.6421976,0.9904709)(0.82,0.4)(0.82,-0.4)
\psbezier[linecolor=black, linewidth=0.04, linestyle=dashed, dash=0.17638889cm 0.10583334cm](6.42,0.8)(5.62,1.2)(4.62,1.0)(4.02,0.8)(3.42,0.6)(2.82,0.0)(2.82,-0.4)
\rput[bl](7.02,0.8){\large $v^0_\alpha$}
\rput[bl](8.02,-0.6){\large $w^0_\alpha$}
\rput[bl](0.62,-2.2){\Large $P^0_\alpha$}
\psline[linecolor=black, linewidth=0.04, linestyle=dashed, dash=0.17638889cm 0.10583334cm, arrowsize=0.05291666666666667cm 2.0,arrowlength=1.4,arrowinset=0.0]{->}(7.42,-0.6)(8.02,-1.6)(9.82,-0.4)
\rput[bl](9.02,-1.8){\Large $Q^0_\alpha$}
\end{pspicture}
}\caption{Extending $P^0_\alpha$ to $P^0_{\alpha+1}$.}
\label{claimfig}
\end{figure}
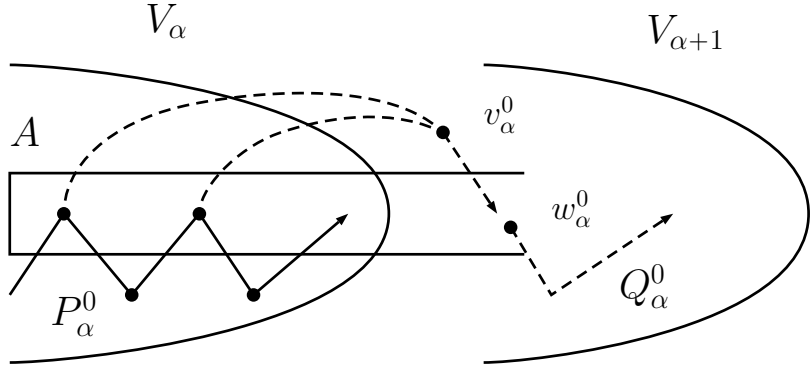

Note that $P^i_\alpha\smf (v^i_\alpha,w^i_\alpha)$ is still a path in color $i$ for $i<2$; see Figure \ref{claimfig}. Now, let us find disjoint sets $Q^0_\alpha,Q^1_\alpha$ such that 
\begin{enumerate}
	\item $Q^i_\alpha$ is a path of color $i$ and order type $\omega$ for $i<2$,
	\item $Q^0_\alpha\cup Q^1_\alpha= (V_{\alpha+1}\setm V_\alpha) \setminus \{v^0_\alpha, v^1_\alpha\}$,
   \item the first point of $Q^i_\alpha$ is $w^i_\alpha$ for $i<2$,
	\item $A\cap Q^0_\alpha$ is cofinal in $Q^0_\alpha$ and $(\omg\setm C)\cap Q^1_{\alpha}$ is cofinal in
$Q^1_\alpha$.
\end{enumerate}
Similarly as above, this is easily done by setting $D = (V_{\alpha+1} \setminus V_\alpha) \setminus \{v^0_\alpha, v^1_\alpha\}$ and applying Lemma \ref{simpath} to $G[D]$ and $C_0 = C \cap D$, $C_1 = (\omega_1 \setminus C) \cap D$, $A_0 = A \cap D$ and $A_1 = \emptyset$. Note that $$P^i_{\alpha+1}=P^i_\alpha\smf (v^i_\alpha)\smf Q^i_\alpha$$ is as desired (for $i<2$).

Finally, let $P^i=\bigcup \{P^i_\alpha:\alpha<\omg\}$ for $i<2$. Then $P^0$ and $P^1$ are monochromatic paths of distinct colors which partition $\omg$.

\textbf{Case 2:} There is no monochromatic copy of $\half$. 

Lemma \ref{dich} implies that any 
uncountable set of vertices must be a trail in both colors. Let us find an uncountable $A\subs V$ which
 is $\omg$-connected in some color by Lemma \ref{uftrick2}. We can suppose that $A$ is $\omg$-connected in color 0 and extend $A$ to a \emph{maximal} $\omg$-connected set $C$ in color $0$.

\begin{claim} $V\setm C$ is countable and $\sat \omg$ in color 1.
\end{claim}
\begin{proof}Indeed, by the maximality of $C$, it is easy to see that $|N(v,0)\cap C|\leq \omega$ for every $v \in V\setminus C$; this immediately gives that $V\setm C$ is $\sat \omg$ in color 1. Moreover, if $V\setm C$ is uncountable then the proof of Claim \ref{halfcover} shows that we can find a monochromatic copy of $\half$ which contradicts our assumption.
\end{proof}

Now cover $V\setm C$ by a path $P^1$ of color 1 and order type $\omega$ using Corollary \ref{trivpath}. 
By assumption, $C\setm P^1$ is still a  trail and remains $\omg$-connected in color 0; that is, $C\setm P^1$ is a path $P^0$ of color 0 by Lemma \ref{trail}. We conclude the proof by noting that $P^0\cup P^1$ is the desired partition.
\end{proof}

\section{Further results and open problems} \label{sc:problems}

In general, there are two directions in which one can aim to extend our results: 
investigate edge colored non-complete graphs; determine the exact number of 
monochromatic structures (paths, powers of paths) needed to cover 
the vertex set of a certain edge 
colored graph. 

First, for state-of-the-art results and problems concerning finite graphs 
 we refer the reader to A. Gy\'arf\'as 
\cite{Gy2}. Second, let us mention some results and problems about countably 
infinite graphs. Let $K_{\omega, \omega}$ denote the complete bipartite graph 
with two countably infinite classes. The following statements 
can be proved very similarly to our proof of Theorem \ref{rado}:

\begin{cclaim}\label{ctblbipartite} Let $c:E(K_{\omega, \omega})\to r$ for some 
$r\in \mbb N$. Then  the vertex set of $K_{\omega, \omega}$ can be partitioned 
into at most $2r-1$ monochromatic paths. Furthermore, for every $r\in \mbb N$ 
there is $c_r:E(K_{\omega, \omega})\to r$ so that the vertex set of $K_{\omega, 
\omega}$ cannot be covered by less than $2r-1$ monochromatic paths.
\end{cclaim}

\begin{cclaim} For every $r$-edge coloring of the random graph on $\mbb N$  we 
can partition the vertices into $r$ disjoint paths of distinct colors.
\end{cclaim}

Regarding Theorem \ref{tm:fonat1} we ask the following most general question:

\begin{prob}
What is the exact number of monochromatic $k^{th}$ powers of paths needed to 
partition the vertices of an $r$-edge colored complete graph on $\mbb N$?
\end{prob}

Naturally, any result aside from the resolved case of $k=r=2$ (see Theorem 
\ref{tm:fonat2}) would be very welcome. In particular:

\begin{prob}
Can we bound the number of monochromatic $k^{th}$ powers of paths needed to 
partition the vertices of an $r$-edge colored complete graph on $\mbb N$ by a 
function of $r$ and $k$?
\end{prob}

Finally, turning to arbitrary infinite complete graphs, we announce the 
following complete solution to Rado's problem from \cite{R}:

\begin{theorem}[D. T. Soukup, \cite{monopath2}] The vertices of a finite-edge colored infinite complete graph can be partitioned into disjoint monochromatic paths of different colour.
\end{theorem}

\section{Acknowledgments}

The authors were supported by the grant OTKA 113047.

The first author was supported by the grant OTKA 104178. The third author was partially supported by the Ontario Trillium Fellowship at the University of Toronto.

\end{document}